\documentclass[final]{siamltex1213}
\usepackage{amssymb}
\usepackage{graphicx}
\usepackage[all]{xy}
\usepackage{color}
\usepackage{latexsym}
\usepackage{amssymb,amsmath,amstext}
\usepackage{epsfig}
\usepackage{graphics}
\usepackage{verbatim}

\newtheorem{example}{Example}
\newtheorem{remark}{Remark}

\newcommand{\C}{\mathbb{C}}
\newcommand{\Z}{\mathbb{Z}}
\newcommand{\PP}{\mathbb{P}}
\newcommand{\R}{\mathbb{R}}

\newcommand{\Q}{\mathbb{Q}}
\newcommand{\sO}{\mathcal{O}}
\allowdisplaybreaks

\title{\bf Exact Solutions in\\
Structured Low-Rank Approximation}

\author{Giorgio Ottaviani\footnotemark[2]\ 
\and Pierre-Jean Spaenlehauer\footnotemark[3]\ 
\and Bernd Sturmfels\footnotemark[4]}

\begin{document}
\maketitle
\renewcommand{\thefootnote}{\fnsymbol{footnote}}
\footnotetext[2]{Universit\`a di Firenze, viale Morgagni 67A, 
50134 Firenze, Italy, {\tt ottavian@math.unifi.it}}
\footnotetext[3]{CARAMEL project, Inria Nancy Grand-Est; Universit\'e de Lorraine; CNRS. LORIA, Nancy, France, {\tt pierre-jean.spaenlehauer@inria.fr}}
\footnotetext[4]{University of California, Berkeley, CA 94720-3840, 
USA, {\tt bernd@berkeley.edu}}

\begin{abstract} \noindent
Structured low-rank approximation is the problem of
minimizing a weighted Frobenius distance to a given  matrix
among all matrices of fixed rank 
in a  linear space of matrices.
We study the critical points
of this optimization problem
using algebraic geometry. 
A particular focus lies on
Hankel matrices, Sylvester matrices and
generic linear spaces.

\end{abstract}

\section{Introduction}
{\em Low-rank approximation} in linear algebra refers to the following optimization problem:
\begin{equation}
\label{eq:frobnorm1}
{\rm minimize} \,\,\,\,
 ||\, X-U\,||_\Lambda^2 \,\,\,= \,\,\,
 \sum_{i=1}^m \sum_{j=1}^n  \lambda_{ij} (x_{ij}  - u_{ij})^2
\,\,\quad  \hbox{subject to} \quad {\rm rank}(X) \leq r .
\end{equation}
Here, we are given a real {\em data matrix} $U = (u_{ij})$ of format $m \times n$, 
and we wish to find a matrix $X = (x_{ij})$ of rank at most $r$ that is closest to $U$ in 
a weighted Frobenius norm. The entries of the {\em weight matrix}
$\Lambda = (\lambda_{ij})$ are positive reals.
If $m \leq n$ and the weight matrix  $\Lambda$ is the all-one matrix ${\bf 1}$ then
the solution to (\ref{eq:frobnorm1}) is given by the singular value decomposition
$$
 U \,\,=\,\,\, T_1 \cdot {\rm diag}(\sigma_1, \sigma_2, \ldots, \sigma_m) \cdot T_2.
$$
Here $T_1,T_2$ are orthogonal matrices, and
 $ \sigma_1 \geq \sigma_2 \geq \cdots \geq \sigma_m $ are the singular values of $U$.
By the  Eckart-Young Theorem, the matrix of rank $\leq r$  closest to
 $U$ equals
\begin{equation}
\label{eq:EY}
 U^* \,\, = \,\,\,                                                                           
T_1 \cdot  {\rm diag}(\sigma_1, \ldots, \sigma_r,0, \ldots, 0) \cdot T_2. 
\end{equation}
For  weights $\Lambda$, the situation is more complicated,
as seen in the studies \cite{MMH, Rey, SJ}. In particular,
there can be many local minima.
We discuss a small instance in  Example~\ref{ex:rey}.

In {\em structured low-rank approximation} \cite{CFP, Mar},
we are also given a  linear subspace $\mathcal{L} \subset 
\R^{m \times n}$, typically containing the matrix $U$.
We consider the restricted problem:
\begin{equation}
\label{eq:frobnorm2}
\! {\rm minimize} \,\, ||\, X-U\,||_\Lambda^2 \,= 
\sum_{i=1}^m \sum_{j=1}^n \lambda_{ij} (x_{ij}  - u_{ij})^2
\,\,\,  \hbox{subj.~to} \,\, X \in \mathcal{L} \,\hbox{and} \, {\rm rank}(X) \leq r .
\end{equation}
A best-case scenario for $\Lambda = {\bf 1}$ is this: if $U $ lies in $\mathcal{L}$ 
then so does  $U^*$.
This happens for some subspaces $\mathcal{L}$, including 
symmetric and circulant matrices, but
most  subspaces $\mathcal{L}$ 
do not enjoy this property (cf.~\cite{CFP}). 
Our problem is difficult  even for $\Lambda = {\bf 1}$. 

Most practitioners use {\em local methods} to solve  (\ref{eq:frobnorm2}).
   These methods return a local minimum. There are many heuristics
for ensuring that a local minimum is in fact
a global minimum, but there is never a guarantee
that this has been accomplished. Another
approach is to set up {\em sum  of squares} relaxations,
which are then solved with semidefinite programming (cf.~\cite{BPT}).
 These SOS methods furnish certificates of global optimality
whenever the relaxation is exact. While this does happen
in many instances, there is no a-priori guarantee either.

How then can one reliably find all global optima to
a polynomial optimization problem such as (\ref{eq:frobnorm2})?
Aside from interval arithmetic and domain decomposition techniques,
the only sure method we are aware of is to 
list and examine all the critical points. Algorithms
that identify the critical points, notably 
{\em Gr\"obner bases} \cite{Fau02}
 and {\em numerical algebraic geometry} \cite{BHSW}, 
 find all solutions over the complex numbers
     and sort out the 
real solutions after the fact.
The number of complex critical points is an intrinsic
invariant of an optimization problem,
and it is a good indicator of the running
time needed to solve that problem exactly.
The study of such algebraic degrees is an active
area of research, and well-developed results are now
available for semidefinite programming \cite{Ran} and
maximum likelihood estimation \cite{CHKS}.

The present paper applies this philosophy
to structured low-rank approximation.
A  general degree theory for  closest points on algebraic varieties
was introduced by Draisma {\it et al.}~in \cite{DHOST}.
Following their approach, our primary task is to compute the
number of complex critical points of (\ref{eq:frobnorm2}).
Thus, we seek to find the
{\em Euclidean distance degree} (ED degree)~of
$$ \mathcal{L}_{\leq r} \quad := \quad
\bigl\{\, X \in \mathcal{L} \,: \, {\rm rank}(X) \leq r \,\bigr\}.$$
This {\em determinantal variety} is always regarded as a subvariety
of the matrix space $\R^{m \times n}$, and we  use the
$\Lambda$-weighted Euclidean distance coming from $\R^{m \times n}$.
We write ${\rm EDdegree}_\Lambda( \mathcal{L}_{\leq r} )$
for the $\Lambda$-weighted Euclidean distance degree
of the variety $\mathcal{L}_{\leq r}$. Thus
${\rm EDdegree}_\Lambda( \mathcal{L}_{\leq r} )$ is the number
of complex critical points of the problem (\ref{eq:frobnorm2})
for generic data matrices $U$.
The importance of keeping track of the
weights $\Lambda$ was highlighted in
\cite[Example 3.2]{DHOST}, for the
seemingly harmless situation when $\mathcal{L}$ is the
subspace of all symmetric matrices in $\R^{n \times n}$.

Our initial focus lies on the 
{\em unit ED degree}, when $\Lambda = {\bf 1}$ is the all-one matrix,
and on the {\em generic ED degree}, denoted
${\rm EDdegree}_{\rm gen}( \mathcal{L}_{\leq r} )$,
when the weight matrix $\Lambda$ is generic.
Choosing generic weights $\lambda_{ij}$ ensures that the variety
${\mathcal L}_{\leq r}$ meets the isotropic quadric
transversally, and it hence allows us to apply
formulas from intersection theory such as \cite[Theorem 7.7]{DHOST}.

This paper is organized as follows.
In Section 2 we offer a computational
study of our optimization problem (\ref{eq:frobnorm2}) 
when the subspace $\mathcal{L}$ is generic of codimension $c$.
Two cases  are to be distinguished: either
$\mathcal{L}$ is a vector space, defined by $c$
homogeneous linear equations in the matrix entries, or
$\mathcal{L}$ is an affine space, defined by $c$
inhomogeneous linear equations.
We refer to these as the  {\em linear case}
and {\em affine case} respectively.
We present Gr\"obner basis methods for computing all
complex critical points, and we report on
their performance. From the complex critical points, one
identifies all real critical points and all local minima.

In Section 3 we  derive some explicit formulas
for ${\rm EDdegree}_{\rm gen}(\mathcal{L}_{\leq r})$
when $\mathcal{L}$ is generic. We cover the four cases
that arise by pairing the affine case and the linear case
with either unit weights or generic weights.
Here we are using techniques from algebraic geometry,
including Chern classes and the analysis of singularities.
In Section 4, we shift gears and we focus on special matrices, namely
Hankel matrices and Sylvester matrices.
Those spaces $\mathcal{L}$ arise naturally 
from  symmetric tensor decompositions and
approximate {GCD} computations.
These applications require the use of
certain specific  weight matrices $\Lambda$
other than ${\bf 1}$.

We close the introduction with two examples
that illustrate the concepts above.

\begin{example} \label{eq:hankel33} \rm
Let $m=n=3$ and
$\mathcal{L} \subset \R^{3 \times 3}$ the 
$5$-dimensional
 space of Hankel matrices:
$$ X \, = \, \begin{bmatrix} x_0 & x_1 & x_2 \\  
                                               x_1 & x_2 & x_3 \\
                                               x_2 & x_3 & x_4 \end{bmatrix} , 
                                               \quad 
U \, = \, \begin{bmatrix} u_0 & u_1 & u_2 \\  
                                               u_1 & u_2 & u_3 \\
                                               u_2 & u_3 & u_4 \end{bmatrix}  \quad \hbox{and} \quad
\Lambda \, = \, \begin{bmatrix} \lambda_0 & \lambda_1 & \lambda_2 \\  
                                               \lambda_1 & \lambda_2 & \lambda_3 \\
                                               \lambda_2 & \lambda_3 & \lambda_4 \end{bmatrix} .
$$
Our goal in  (\ref{eq:frobnorm2}) 
is to solve the following constrained optimization problem
for $r = 1,2$:
$$ \begin{matrix} {\rm minimize} \,\,
\lambda_0 (x_0 - u_0)^2\!+\! 
2 \lambda_1 (x_1 - u_1)^2 \!+\! 
3 \lambda_2 (x_2 - u_2)^2 \!+\! 
2 \lambda_3 (x_3 - u_3)^2 \!+\! 
\lambda_4 (x_4 - u_4)^2  \\
\, \hbox{subject to} \,\, {\rm rank}(X) \leq r .
\end{matrix}
$$
This can stated as an unconstrained optimization problem.
For instance, 
for rank $r = 1$, we get a one-to-one parametrization  of
$\mathcal{L}_{\leq 1}$ by setting
$x_i = s t^i$, and we seek to
$$ {\rm minimize} \,\,
\lambda_0 (s- u_0)^2 + 
2 \lambda_1 (st - u_1)^2 + 
3 \lambda_2 (st^2 - u_2)^2 + 
2 \lambda_3 (st^3- u_3)^2 + 
\lambda_4 (st^4- u_4)^2  .
$$
The ED degree is the number of critical points
with $t \not= 0$.
 We consider three weights:
$$ 
{\bf 1} \, = \,
\begin{bmatrix} 1  & 1 & 1 \\
1 & 1 & 1 \\
1 & 1 & 1 
\end{bmatrix} , \qquad
\Omega \, = \,
\begin{bmatrix} 1 & 1/2 & 1/3 \\
1/2 & 1/3 & 1/2 \\
1/3  & 1/2 & 1 
\end{bmatrix} , \qquad
\Theta  \, = \,
\begin{bmatrix} 1  & 2 & 2 \\
2 & 2 & 2 \\
2 & 2 & 1 
\end{bmatrix} .
$$
Here $\Omega$  gives the usual Euclidean metric when
$\mathcal{L}$ is identified with $\R^5$, and 
$\Theta$ arises from identifying $\mathcal{L}$
with symmetric $2 \times 2 \times 2 \times 2$-tensors,
as in Section 4. We compute
$$ \begin{matrix}
{\rm EDdegree}_{\bf 1}(\mathcal{L}_{\leq 1}) = 6, & \quad &
{\rm EDdegree}_\Omega (\mathcal{L}_{\leq 1}) = 10, & \quad &
{\rm EDdegree}_\Theta(\mathcal{L}_{\leq 1}) = 4 , \\
{\rm EDdegree}_{\bf 1}(\mathcal{L}_{\leq 2}) = 9, & \quad &
{\rm EDdegree}_\Omega (\mathcal{L}_{\leq 2}) = 13, & \quad &
{\rm EDdegree}_\Theta(\mathcal{L}_{\leq 2}) = 7 . \\
\end{matrix}
$$
In both cases, $\Omega$ exhibits the generic behavior:
${\rm EDdegree}_{\rm gen}(\mathcal{L}_{\leq r}) = {\rm EDdegree}_\Omega(\mathcal{L}_{\leq r})$. 
See Sections 3 and 4 for larger Hankel matrices
and  formulas for their ED degrees.
\hfill $\diamondsuit$
\end{example}

\begin{example} \label{ex:rey} \rm
Let $m=n=3, r = 1$ but now take $\mathcal{L} = \R^{3 \times 3}$,
so this is just the weighted rank-one approximation
problem for $3 \times 3$-matrices.
We know from \cite[Example 7.10]{DHOST} that
${\rm EDdegree}_{\rm gen}(\mathcal{L}_{\leq 1}) = 39$.
We take a circulant data matrix and a circulant weight matrix:
$$ U=\begin{bmatrix}
-59& \phantom{-}11& \phantom{-} 59 \, \\
\phantom{-}11& \phantom{-}59 & -59 \, \\
\phantom{-}59&-59& \phantom{-}11 \,
\end{bmatrix}
\quad \hbox{and} \quad
 \Lambda =\begin{bmatrix}
 9 &  6 &  1 \\
  6 &  1 & 9 \\
  1 & 9 & 6 
  \end{bmatrix}.
  $$
  This instance has $39$  critical points.
  Of these,  $19$ are real, and $7$ are local minima:
  \begin{small}  $$
\begin{bmatrix}
                 \phantom{-}0.0826 \! & \!    \phantom{-}2.7921 \! & \!     -1.5452 \\
                 \phantom{-}2.7921 \! & \!   \phantom{-}94.3235 \! & \!     -52.2007 \\
                -1.5452 \! & \!  -52.2007 \! & \!   \phantom{-}28.8890 
\end{bmatrix} \! ,
\begin{bmatrix}
      -52.2007 \! & \!   \phantom{-}28.8890 \! & \!    -1.5452 \\
        \phantom{-}2.7921 \! & \!   -1.5452 \! & \!     \phantom{-}0.0826 \\
       \phantom{-}94.3235 \! & \!  -52.2007 \! & \!     \phantom{-}2.7921
\end{bmatrix} \!$$ 
$$\begin{bmatrix}
       -52.2007 \! & \!   \phantom{-}2.7921 \! & \!    \phantom{-}94.3235 \\
        \phantom{-}28.8890 \! & \!  -1.5452 \! & \!   -52.2007 \\
        -1.5452 \! & \!   \phantom{-}0.0826 \! & \!    \phantom{-}2.7921 
\end{bmatrix} \!,
\begin{bmatrix}
  -29.8794 \! & \!    \phantom{-}36.2165 \! & \!    -27.2599 \\
  -32.7508 \! & \!    \phantom{-}39.6968 \! & \!    -29.8794 \\
   \phantom{-}39.6968 \! & \!   -48.1160 \! & \!    \phantom{-}36.2165
\end{bmatrix}$$
$$\begin{bmatrix}
  -48.1160 \! & \!   \phantom{-}36.2165 \! & \!     \phantom{-}39.6968 \\
   \phantom{-}36.2165 \! & \!   -27.2599 \! & \!    -29.8794 \\
   \phantom{-}39.6968 \! & \!   -29.8794 \! & \!    -32.7508
\end{bmatrix} \!\! , \!
\begin{bmatrix}
  -29.8794 \! & \!   -32.7508 \! & \!    \phantom{-}39.6968 \\
   \phantom{-}36.2165 \! & \!    \phantom{-}39.6968 \! & \!    -48.1160 \\
  -27.2599 \! & \!   -29.8794 \! & \!     \phantom{-}36.2165 
\end{bmatrix}
$$
$$
\begin{bmatrix}
 -25.375 &    -25.375 &   -25.375 \\
 -25.375 &    -25.375 &   -25.375 \\
 -25.375 &    -25.375 &   -25.375
\end{bmatrix}.
$$
\end{small}
The first three are the global minima.
The last matrix is the local minimum where the objective function has the largest value:
note that each entry  equals $-203/8$.
The entries of the first six matrices are algebraic numbers of
degree $10$ over $\mathbb{Q}$. For instance,
the two upper left entries  $0.0826$ and   $-48.1160$
are among the four real roots of the irreducible polynomial
$$
\begin{small} 
\begin{matrix}
164466028468224 x^{10}+27858648335954688 x^9+1602205386689376672 x^8\\+
7285836260028875412 x^7
-2198728936046680414272 x^6\\-14854532690380098143152 x^5
+
2688673091228371095762316 x^4\\+44612094455115888622678587 x^3 
-
41350080445712457319337106 x^2
\\+27039129499043116889674775 x-1977632463563766878765625.
\end{matrix}
\end{small}
$$
Thus, the critical ideal in $\mathbb{Q}[x_{11}, x_{12}, \ldots, x_{33}]$ is not prime.
It is the intersection of six maximal ideals. Their
degrees over $\mathbb{Q}$ are $ 1,2,6,10,10,10$, for a total of  $39
= {\rm EDdegree}_{\rm gen}(\mathcal{L}_{\leq 1})$.
\hfill $\diamondsuit $
\end{example}

    William Rey  \cite{Rey}
reports on numerical experiments with the  optimization problem (\ref{eq:frobnorm1}),
and he asks whether the number of local minima is 
bounded above by ${\rm min}(m,n)$.
Our Example \ref{ex:rey} gives a negative answer: the number of 
local minima can exceed $\min(m, n)$.
This result highlights the value of
our exact algebraic methods for practitioners of optimization.

\section{Gr\"obner Bases}

The critical points of the low-rank approximation problem (\ref{eq:frobnorm2})
can be computed as the solution set of a system of polynomial equations.
In this section we derive these equations, and we demonstrate how
to solve a range of instances using current Gr\"obner basis techniques.
Here, our emphasis lies on the case when $\mathcal{L}$ is a generic
subspace, either linear or affine.

Starting with the linear case, let
$\{L_1,L_2,\ldots,L_s\}$ be a basis of $\mathcal{L}^\perp$, the space of linear forms on
$\R^{m \times n}$ that vanish on $\mathcal{L}$. Thus  ${\rm codim}(\mathcal{L}) = s$,
each derivative $\partial L_k / \partial x_{ij}$ is a constant, and 
$\mathcal{L} = \{X \in \R^{m \times n} : L_1(X) = \cdots = L_s(X) = 0 \}$.
The case when $\mathcal{L}$ is an affine space can be treated with the
same notation if we take
each  $L_i$ to be a linear form plus a constant.

The following implicit formulation of the critical equations is a variation on \cite[(2.1)]{DHOST}.
We begin with the case 
$m = n = r+1$. Let $D\in\Z[x_{11},\ldots,x_{nn}]$ denote the determinant of the 
$n \times n$-matrix $X = (x_{ij})$. Given
a data matrix $\,U = (u_{ij}) \in\mathcal \R^{n\times n}$, the critical points of 
 $\sum_{i=1}^n\sum_{j=1}^n \lambda_{ij} (x_{ij}-u_{ij})^2$ on
 the determinantal hypersurface
 $\mathcal{L}_{\leq n-1} = \{X \in \mathcal{L} : D(X) = 0\}$ verify the following conditions.
 The matrix on the right has $s+2$ rows and $n^2$ columns:
$$\begin{cases}D(X)=0\\
L_1(X)=0\\
\quad\quad\vdots\\
L_s(X)=0
\end{cases}
\quad\quad\quad\quad\mathsf{Rank}
\begin{bmatrix}
\partial D/ \partial  x_{11}&\cdots& \partial D/ \partial    x_{nn}\\ 
\partial L_1/\partial x_{11}& \cdots & \partial  L_1 / \partial x_{nn} \\
\vdots&\ddots&\vdots\\
\partial L_s / \partial x_{11}&\cdots& \partial  L_s / \partial x_{nn} \\
\lambda_{11} (x_{11}-u_{11}) &\cdots& \lambda_{nn}(x_{nn}-u_{nn})
\end{bmatrix}
\leq s+1.$$

Any singular point of $\mathcal L_{\leq n-1}$ also satisfies these conditions.
The rank condition on the Jacobian matrix can be modeled by
introducing Lagrange multipliers $z_0,z_1,\ldots,z_s$.
These are new variables.
We now consider the following polynomial system 
in $n^2+s+1$ variables:
\begin{equation}\label{eq:EDdeterminant}
\!\!\! \begin{cases}D(X)=0\\
L_1(X)=0\\
\quad \vdots\\
L_s(X)=0
\end{cases}
 \begin{bmatrix} z_0 \! \! & \! \!
\cdots \! \! &\!\! z_s \!\! & \! \! 1 
\end{bmatrix}
\! \cdot \!
\begin{bmatrix}
\partial D/ \partial  x_{11}&\cdots& \partial D/ \partial    x_{nn}\\ 
\partial L_1/\partial x_{11}& \cdots & \partial  L_1 / \partial x_{nn} \\
\vdots&\ddots&\vdots\\
\partial L_s / \partial x_{11}&\cdots& \partial  L_s / \partial x_{nn} \\
\lambda_{11}(x_{11}{-}u_{11}) \! &\cdots& \! \lambda_{nn}(x_{nn}{-}u_{nn})
\end{bmatrix}
 =  \begin{bmatrix}  0 \!\! &\!\! \cdots \!\! & \!\! 0 \end{bmatrix}.
\end{equation}
 Table \ref{table:EDdeterminant} shows
the number of complex solutions to these equations.
These numbers are obtained from the formulas in
Section 3. We verified them using Gr\"obner bases.

\begin{table}
\centering
\begin{tabular}{|c|c|c|c|c|}
\hline
\multicolumn{5}{|c|}{linear, $\Lambda = {\bf 1}$}\\
\hline
$\,\quad n=$ &$2$&$3$&$4$&$5$\\
\hline\hline
$\! s=0 \!$&2&3&4&5\\
$s=1$&4&15&28&45\\
$s=2$&2&31&92&205\\
$s=3$&0&39&188&605\\
$s=4$&&33&260&1221\\
$s=5$&&21&284&1805\\
$s=6$&&9&284&2125\\
$s=7$&&3&284& 2205 \\
$s=8$&&0&284&2205 \\
$s=9$&&&264&2205 \\
$s=10$&&&204&2205 \\
$s=11$&&&120&2205 \\
$s=12$&&&52& 2205 \\
$s=13$&&&16& 2205\\
$s=14$&&&4& 2205 \\
$s=15$&&&0&2205\\
\hline
\end{tabular}
\begin{tabular}{|c|c|c|c|c|}
\hline
\multicolumn{4}{|c|}{affine, $\Lambda = {\bf 1}$}\\
\hline
$2$&$3$&$4$&$5$\\
\hline\hline
2&3&4&5\\
6&15&28&45\\
4&31&92&205\\
2&39&188&605\\
&39&260&1221\\
&33&284&1805\\
&21&284&2125\\
&9&284&2205 \\
&3&284&2205 \\
&&284&2205 \\
&&264&2205 \\
&&204& 2205 \\
&&120& 2205 \\
&&52&2205 \\
&&16&2205 \\
&&4&2205\\
\hline
\end{tabular}
\begin{tabular}{|c|c|c|c|c|c|}
\hline
\multicolumn{5}{|c|}{linear, $\Lambda$ gen.} \\
\hline
$\,\quad n=$ &$2$&$3$&$4$&$5$\\
\hline\hline
$\! s=0 \!$&6&39&284 & 2205  \\
$s=1$&4&39&284 & 2205 \\
$s=2$&2&39&284 & 2205 \\
$s=3$&0&39&284 & 2205 \\
$s=4$&&33&284 & 2205 \\
$s=5$&&21&284 & 2205 \\
$s=6$&&9&284 & 2205 \\
$s=7$&&3&284 & 2205 \\
$s=8$&&0&284 & 2205 \\
$s=9$&&&264 & 2205 \\
$s=10$&&&204 & 2205 \\
$s=11$&&&120 & 2205 \\
$s=12$&&&52 & 2205 \\
$s=13$&&&16 & 2205 \\
$s=14$&&&4 & 2205 \\
$s=15$&&&0 & 2205 \\
\hline
\end{tabular}
\begin{tabular}{|c|c|c|c|c|}
\hline
\multicolumn{4}{|c|}{affine, $\Lambda$ gen.}\\
\hline
$2$&$3$&$4$&$5$\\
\hline\hline
6&39&284&2205 \\
6&39&284&2205  \\
4&39&284&2205  \\
2&39&284&2205  \\
&39&284&2205  \\
&33&284&2205  \\
&21&284&2205  \\
&9&284&2205  \\
&3&284&2205  \\
&&284&2205  \\
&&264&2205  \\
&&204&2205  \\
&&120&2205  \\
&&52&2205  \\
&&16&2205  \\
&&4&2205  \\
\hline
\end{tabular}
\caption{\label{table:EDdeterminant}
The ED degree for the determinant of an  $n \times n$-matrix
with linear or affine entries.}
\end{table}

We observe that Table~\ref{table:EDdeterminant} has the following remarkable properties:
\begin{itemize}
\item{} There is a shift  between the ED degrees of affine and linear sections
 for $s\geq (n-1)^2$. This phenomenon will be explained 
in Proposition \ref{prop:affinelinear}.
\item{} For $\Lambda$ general, the third block of columns (linear entries) is constant for $s\le n(n-2)$ and the fourth one (affine entries) is constant for $s\le n(n-2)+1$.
This is explained in Corollaries \ref{thm:sEDpolar} and \ref{cor:EDsequal2}.
\item{} The differences between the first and the third block of columns 
(both with linear entries) equal those between the second and the fourth one
(both with affine entries). This gap
is expressed (conjecturally) with formula~(\ref{discrepancy}).
\end{itemize}
We prove the correctness of the formulation  (\ref{eq:EDdeterminant})
and then  discuss our computations.

\begin{proposition} \label{prop:formulation1}
For a generic linear (or affine) space $\mathcal L$ of codimension $s$ and for a generic 
data matrix $ U = (u_{ij})$ in $\mathcal L$, the solutions
$(X,z)$ of the polynomial system (\ref{eq:EDdeterminant}) correspond to the critical points 
$X$ of the optimization problem (\ref{eq:frobnorm2})
for square matrices of corank one.
\end{proposition}

\begin{proof}
We prove this for linear spaces $\mathcal{L}$. The argument is similar when 
$\mathcal L$ is an affine space.
Any solution of the system (\ref{eq:EDdeterminant}) corresponds to a point of $\mathcal L$ where the Jacobian matrix of $(D, L_1,\ldots, L_s, \lVert X-U \rVert_\Lambda^2)$ has a rank defect. There
are two types of such points: 
the critical points of the distance function and singular points on the determinantal variety. 
Hence  it  suffices to prove that no point in the singular locus corresponds to a solution of (\ref{eq:EDdeterminant}).
The matrix $U=(u_{ij})_{1\leq i,j\leq n}$  was assumed to be generic, so it has rank~$n$ 
since $\mathcal L$ is also generic.

If $X$ is a singular point of the linear section of the 
variety defined by $D(X)=L_1(X)=\dots=L_s(X)=0$, then there exists $(y_0,y_1,\ldots,y_s) $ with $y_0 \neq 0$ such that
$$\begin{bmatrix}y_0 & y_1 & \cdots & y_s 
\end{bmatrix}\cdot
\begin{bmatrix} \partial D/  \partial  x_{11}&\cdots&     \partial D/ \partial x_{nn} \\
\partial L_1/\partial x_{11} &\dots& \partial   L_1/\partial x_{nn} \\
\vdots&\ddots&\vdots\\
\partial L_s / \partial x_{11} &\cdots&
\partial  L_s/\partial x_{nn}
\end{bmatrix}
\, = \, \begin{bmatrix}  0 \! &\! \cdots \! & \! 0 \end{bmatrix}. $$
Let us assume by contradiction that $X$ extends to a solution
$(X,z)$ of (\ref{eq:EDdeterminant}).
Then
$$
\begin{bmatrix}
0\!&
\!(z_1\!-\! \dfrac{y_1 z_0}{y_0})\!
& \!\!\! \cdots\!\!\! &
(z_s\!-\! \dfrac{y_s z_0}{y_0})\!
&\!1
\end{bmatrix}
\cdot
\begin{bmatrix}
\partial D/ \partial  x_{11}&\! \cdots \!& \partial D/ \partial    x_{nn}\\ 
\partial L_1/\partial x_{11}&\! \cdots \! & \partial  L_1 / \partial x_{nn} \\
\vdots&\! \ddots \!&\vdots\\
\partial L_s / \partial x_{11}&\! \cdots \!& \partial  L_s / \partial x_{nn} \\
\lambda_{11}(x_{11}{-}u_{11}) \! &\! \cdots\!& \! \lambda_{nn}(x_{nn}{-}u_{nn})
\end{bmatrix}
 =  \begin{bmatrix}  0 \! &\! \cdots \! & \! 0 \end{bmatrix}.
$$
This means that $X-U$ belongs to $\mathcal L$ and $\Lambda*(X-U)$ belongs to $\mathcal L^\perp$. 
Here $*$ denotes the Hadamard (coordinatewise) product of two matrices.
The scalar product of $X-U$ and $\Lambda*(X-U)$ is zero.
Since all coordinates live in $\mathbb{R}$,
these conditions imply
$\lVert X-U \rVert_\Lambda^2=0$, and hence $X=U$.  We get a contradiction since $U$ has full rank,
whereas $D(X) = 0$.
\end{proof}

The values of ${\rm EDdegree}_\Lambda(\mathcal{L}_{\leq n-1})$ in
 Table~\ref{table:EDdeterminant} can be verified
 computationally with the formulation
(\ref{eq:EDdeterminant}).
We used the implementation of Faug\`ere's Gr\"obner basis algorithm
$F_5$ \cite{Fau02} in the {\tt maple} package {\tt FGb}.
Computing Gr\"obner bases for  (\ref{eq:EDdeterminant})
was fairly easy for $n \leq 4$, but 
difficult already  for $n=5$.
 For each of the cases in Table~\ref{table:EDdeterminant},
we computed the ED degree by running {\tt FGb} over
the finite field with $65521$ elements.
However, due to substantial coefficient growth, this did not work  over the field
$\mathbb{Q}$ of rational numbers.
Hence, to actually compute all critical points over $\C$
and hence all local minima over $\R$, even for $n = 4$,
a better formulation was required.
In what follows we shall present two such improved formulations.

Duality plays a key role in the computation of the critical points of
the Euclidean distance and was investigated in \cite[\S 5]{DHOST}.  
In what follows, we  compute
the critical points of the weighted Euclidean distance of the
determinant by using this duality.
In the following statement we are using the standing
hypothesis that all $\lambda_{ij}$ are non-zero.

\begin{proposition}\label{prop:weightedDuality}
Let $U$ be a generic $m\times n$ matrix with $m\leq n$,
let $\Lambda$ be a weight matrix,
and fix an integer $r\leq \min(m,n)$. Then there is a bijection between the critical points of 
\begin{enumerate}
\item[(1)] $Q(X)=\sum_{i,j} \lambda_{ij}(x_{ij}-u_{ij})^2$ on the variety $\C^{m\times n}_{\leq m-r}$ of corank $r$ matrices $X$, and
\item[(2)] $Q_{\rm dual}(Y)=\sum_{i,j}
    (y_{ij}{-}\lambda_{ij} u_{ij})^2/\lambda_{ij}$ on the variety $\C^{m\times n}_{\leq r}$ of rank $r$ matrices~$Y$.
\end{enumerate}
For each critical point $X$ of (1), the corresponding critical point $Y$ of (2) 
equals $\,Y=\Lambda*U-\Lambda*X$, where $*$ denotes the
 Hadamard  product. In particular, if $U$ has real entries, then
the bijection interchanges the real critical points of (1) and of (2).
\end{proposition}

\begin{proof}
The critical points of (1) correspond to matrices $X$ such that
the Hadamard product $\Lambda*(U-X)$ is perpendicular to the tangent space
at $X$ of the variety $\C^{m\times n}_{\leq m-r}$ of corank $r$ matrices.
Recall, e.g.~from  \cite[\S 5]{DHOST},
 that the dual variety to $\C^{m\times n}_{\leq m-r}$ is the variety $\C^{m\times n}_{\leq r}$ of rank $r$ matrices.
Hence, the critical points in (1) can be found by solving
 the linear equation $Y=\Lambda*(U-X)$ 
on the conormal variety.
That  {\em conormal variety} is the set of all 
 pairs $(X,Y)$ such that 
$X\in\C^{m\times n}_{\leq m-r}$, $Y\in\C^{m\times n}_{\leq r}$,
$X^{\rm t} \cdot Y = 0$, and 
$X \cdot Y^{\rm t} = 0$.
We can  now express $X$ in terms of $Y$ and the parameters
by writing $X=\Lambda^{*-1}*(\Lambda*U-Y)$,
where $\Lambda^{*-1}$ denotes the Hadamard (coordinatewise)
inverse of the weight matrix $\Lambda$.
 Using biduality, this means that  
 $X=\Lambda^{*-1}*(\Lambda*U-Y)$ is
 perpendicular to the tangent space at $Y$
of the variety $\C^{m\times n}_{\leq r}$. 
This is equivalent to the statement that $Y$ is a critical point of (2) on $\C^{m\times n}_{\leq r}$.
\end{proof}

In both Propositions \ref{prop:formulation1} and \ref{prop:weightedDuality},
it is assumed that the given matrix $U$ is generic. Here the term {\em generic}
is meant in the usual sense of algebraic geometry: $U$ lies in the complement
of an algebraic hypersurface.
In particular, that complement is dense 
  in $\mathbb{R}^{m \times n}$, so $U$ will be generic
  with probability one when drawn from a probability measure 
supported on  $\mathbb{R}^{m \times n}$.
However, an exact characterization of genericity is difficult.
The polynomial that defines the aforementioned hypersurface
is the {\em ED discriminant}. As can be seen in \cite[\S 7]{DHOST},
this is a very large polynomial of high degree, and we will rarely
be able to identify it in an explicit way.
  
\smallskip

Proposition \ref{prop:weightedDuality} shows that weighted low-rank approximation can be solved by the dual problem. We focus now on the corank 1 case (whose dual problem is rank 1 approximation).
For this, we use the parametrization of $n{\times} n$ matrices of rank $1$ by
\begin{equation}
\label{eq:pararank1}
(t_1,\ldots,t_n, z_1,\ldots, z_{n-1})
\,\,\mapsto \,\, \begin{bmatrix} t_1 & t_1 z_1 &\dots&t_1 z_{n-1}\\
\vdots&\vdots&\vdots&\vdots\\
t_n & t_n z_1 &\dots&t_n z_{n-1}
\end{bmatrix}.
\end{equation}

\begin{remark}\label{differentpara:rmk}
This parametrization is not surjective: the rank $1$ matrices
  whose first column is zero are missing. This is not an issue 
  when $U$ and $\mathcal L$ are generic, since in that case all critical points
  are in the image of the parametrization. However, for specific $U$
  or $\mathcal L$, if some of the critical points are missing, they
  can be computed by choosing $n$ such parametrizations whose ranges
  cover all rank $1$ matrices. This multiplies the computation time   by $n$.
Our {\it a priori} computation of the ED degree is useful also to overcome these difficulties.
Suppose the expected number of critical points
is known. Then, after some parametrizations
have been tried for the given data $(U,\mathcal{L})$,
the user is guaranteed that all critical points have been found.
\end{remark}

The parametrization (\ref{eq:pararank1})
 expresses the dual problem (for corank one)
 as an unconstrained optimization problem in $2n-1$ variables: 
\begin{equation}\label{eq:unconstrOpt}
{\rm Maximize} \,\,
Q_{\rm dual}\,=\! \sum_{1\leq i,j\leq n} \dfrac 1{\lambda_{ij}}(y_{ij}-\lambda_{ij}
  u_{ij})^2, \, \text{ where }y_{i1}=t_i \text{ and } y_{ij}=t_i z_{j-1}.
\end{equation}
Here,  ``maximize''  is used in an unconventional way: 
what we seek is the critical point furthest to $U$. 
 That critical point need not be a local maximum; see e.g.~\cite[Figure 4]{DHOST}.
 We compute the critical points for (\ref{eq:unconstrOpt}) by
applying Gr\"obner bases to the  equations
$$\partial Q_{\rm dual}/\partial t_i\,=\,\partial Q_{\rm dual}/\partial z_j\,=\,0 
\qquad \hbox{ for  } i\in\{1,\ldots,n\} \,\,\hbox{and} \, \, j\in\{1,\ldots, n-1\}.
$$
The critical points of the primal problem are  found by the formula
 $Y=\Lambda*(U-X)$.
 
 \smallskip
 
This concludes our discussion of square matrices of rank $1$ or corank $1$.
We next consider the general case of rectangular matrices
of format $m \times n$ with general linear or affine entries. 
We assume $r \leq m \leq n$ and $s \leq mn$.
Let $M$ be a complex $m \times n$-matrix of rank $r$.
Then $M$ is a smooth point in the variety
 $\C^{m\times n}_{\leq r}$ of matrices of rank $\leq r$.
Let   $\mathsf{Ker}_L(M)$ and $\mathsf{Ker}_R(M)$ denote the left and right kernels of $M$
respectively. The 
 normal space of $\C^{m\times n}_{\leq r}$ at $M$ 
 has dimension  $(m-r)(n-r)$, and it equals
$\mathsf{Ker}_L(M)\otimes \mathsf{Ker}_R(M)
\,\subset \, \C^{m \times n}$ \cite[Chapter 6]{GG}.
Its orthogonal complement  is the  tangent space at $M$,
which has dimension $rm+rn-r^2$.

In order to construct a polynomial system whose solutions are the critical points of
$X \mapsto || X-U||_\Lambda^2$
  on the smooth locus of $\mathcal L_{\leq r}$, we introduce two matrices of unknowns:
  $$Y=\begin{bmatrix}
1&\dots&0\\
\vdots&\ddots&\vdots\\
0&\dots&1\\
y_{1,1}&\dots&y_{1,m-r}\\
\vdots&\ddots&\vdots\\
y_{r,1}&\dots&y_{r,m-r}
\end{bmatrix} \quad\quad \hbox{and} \quad\quad Z=\begin{bmatrix}
1&\dots&0\\
\vdots&\ddots&\vdots\\
0&\dots&1\\
z_{1,1}&\dots&z_{1,n-r}\\
\vdots&\ddots&\vdots\\
z_{r,1}&\dots&z_{r,n-r}
\end{bmatrix}.$$
For $i\in\{1,\ldots, m-r\}$, $j\in\{1,\ldots, n-r\}$, let $N^{((m-r)(j-1)+i)}$ be the rank $1$ 
matrix which is the product of the $i$th column of $Y$ and of the $j$th row of $Z^\intercal$.
We consider 
\begin{equation}\label{eq:EDdetgeneral}
\!\!\! \begin{cases}
Y^\intercal \cdot X = 0\\
X\cdot Z = 0\\
L_1(X) = 0\\
\quad\quad\vdots\\
L_s(X)=0
\end{cases} \begin{bmatrix}w_1 \!\! & \!\!  \cdots \!\! & \!\! w_{(m-r)(n-r)+s} \!\! &\! 1
\end{bmatrix}
\begin{small} \!\!
\begin{bmatrix}
  N^{(1)}_{11}&\!\! \dots \!\!& N^{(1)}_{mn}\\
\vdots&\!\! \ddots \!\!&\vdots\\
  N^{((m-r)(n-r))}_{11}\! &\!\!\dots\!\!& N^{((m-r)(n-r))}_{mn}\\
\partial L_1 / \partial x_{11}& \!\! \cdots \!\! &
\partial  L_1 / \partial x_{mn }\\
\vdots&\!\! \ddots \!\! &\vdots\\
\partial L_s / \partial x_{11 }&\!\! \cdots \!\! &
\partial L_s / \partial x_{mn }\\
\lambda_{11}(x_{11}{-}u_{11})&\!\! \dots \!\! &\lambda_{mn}(x_{mn}{-}u_{mn})
\end{bmatrix}
\end{small}
=0.
\end{equation}

The rank condition on the matrix in (\ref{eq:EDdetgeneral}) comes from the fact that $M\in \mathcal L_{\leq r}$ is a critical point if the gradient of the distance function at $M$ belongs to the normal space of $\mathcal L_{\leq r}$ at $M$. The first $(m-r)(n-r)+s$ rows of the matrix 
span the normal space of $\mathcal L_{\leq r}$ at a smooth point.
This formulation avoids saturating by the singular locus, which is often too costly.

\begin{proposition}\label{prop:generalformulation}
  For a generic affine space $\mathcal L$ of codimension $s$ and a generic matrix $U$ in $\mathcal L$, the polynomial system \ref{eq:EDdetgeneral} has finitely many complex solutions which correspond to the critical points of the weighted Euclidean distance function on the smooth locus of $\mathcal L_{\leq r}$.
\end{proposition}

\begin{proof}
This  is derived  from \cite[Lemma 2.1]{DHOST}.
It is  analogous to Proposition  \ref{prop:formulation1}.
\end{proof}

As in the corank $1$ case, for special data $(U, \mathcal{L})$
some critical points may be missed because our formulation
  computes only the critical points  in a dense open subset of
  $\mathcal L_{\leq r}$.   However, the same fix as in 
  Remark \ref{differentpara:rmk} works here. We can 
  redo the computations in any of the
  $\binom{n}{r}\binom{m}{r}$ charts corresponding to the
  invertibility of pairs of square submatrices of $Y$ and $Z$.

We next discuss our computational experience with Gr\"obner bases.
In Table \ref{table:expeFormulation},
 we compare the efficiency of the different approaches on a specific problem: computing the weighted rank $3$ approximation of a $4\times 4$ matrix. The experimental setting is the following: we consider a $4\times 4$ matrix $U$ with integer entries picked uniformly at random in $\{-100,\ldots, 100\}$ and a random weight matrix $\Lambda$ with positive integer entries chosen at random in $\{1,\ldots, 20\}$. 
 By Table \ref{table:EDdeterminant}, the generic ED degree is $284$ and the 
  ED degree for $\Lambda = {\bf 1}$ is $4$. We report in Table \ref{table:expeFormulation} the timings for computing a lexicographical Gr\"obner basis with the \texttt{maple} package \texttt{FGb} \cite{Fau02}. Once a Gr\"obner basis is known, isolation techniques may be used to obtain the real roots. The \texttt{maple} package \texttt{fgbrs} provides implementations of such methods.

\begin{table}[h]
\centering
\begin{tabular}{@{}|c|@{}c@{}|@{}c@{}|@{}c@{}|@{}c@{}|@{}}
\cline{2-5}
\multicolumn{1}{c|}{}&\begin{tabular}{c}\!Determinant\!\\
primal (\ref{eq:EDdeterminant})
 \end{tabular}
&\begin{tabular}{c}\!Parametric\!\\
 dual (\ref{eq:unconstrOpt})\end{tabular}
&\begin{tabular}{c}\!Normal space\!\\primal (\ref{eq:EDdetgeneral})\end{tabular}
&\begin{tabular}{c}\!Normal space\!\\dual (\ref{eq:EDdetgeneral}) \! \end{tabular}\\
\hline
$\Lambda $ generic, $\mathsf{GF}(65521)$&5s&{\bf 1.3s}&6s&8.6s\\
$\Lambda $ generic, over $\Q$&$>1$ day&{\bf 891s}&1327s&927s\\
$\Lambda = {\bf 1}$, over $\Q$&0.3s&{\bf 0.2s}&0.4s&0.5s\\
\hline
\end{tabular}
\caption{\label{table:expeFormulation} Symbolic computation of the
weighted rank $3$ approximations of a $4\times 4$ matrix}
\end{table}

We examine three scenarios. In the first row, the computation is performed over a finite field.
This gives  information about the algebraic difficulty of the problem: there is no 
coefficient growth, and the timings 
  indicate the number of arithmetic operations in Gr\"obner bases algorithms. However, 
  finding local minima  requires computing   over $\Q$.
 In rows 2 and 3 of   Table \ref{table:expeFormulation}, we compare the
 case of generic weights with the  unweighted case
(\ref{eq:EY}) that corresponds to     the singular value decomposition ($\Lambda=\mathbf 1$).
The dual problem is easiest to solve, in particular with the 
unconstrained formulation  (\ref{eq:unconstrOpt}). Note that, for
$s \geq 1$, such an unconstrained formulation is not available, since
 $\mathcal{L}_{\leq r}$ is generally not a unirational variety.

\begin{table}[h] \centering
\begin{tabular}{|c|c|c|c|c|}
\hline
$(m,n,r)$&$s=0$&$s=1$&$s=2$&$s=3$\\
\hline\hline
$(4,4,2)$&
$\mathbf{4}$/0.42s/1.8s&
$\mathbf{54}$/1.93s/744s&
$\mathbf{230}$/52.7s/--&
$\mathbf{582}$/349.2s/--\\\hline
$(3,4,2)$&
$\mathbf{3}$/0.2s/0.3s&
$\mathbf{15}$/0.3s/7.4s&
$\mathbf{43}$/1.2s/132s&
$\mathbf{71}$/1.5s/1120s\\\hline
$(3,5,2)$&
$\mathbf{3}$/0.3s/0.5s&
$\mathbf{15}$/0.5s/16s&
$\mathbf{43}$/2.1s/400s&
$\mathbf{87}$/7.1s/6038s\\\hline
\end{tabular}

\medskip

\begin{tabular}{|c|c|c|c|}
\hline
$(m,n,r)$&$s=4$&$s=5$&$s=6$\\
\hline\hline
$(4,4,2)$&
$\mathbf{998}$/1474s/--&
$\mathbf{1250}$/2739s/--&
$\mathbf{1250}$/2961s/--\\\hline
$(3,4,2)$&
$\mathbf{83}$/2.2s/2696s&
$\mathbf{83}$/2.3s/4846s&
$\mathbf{83}$/2.1s/5764s\\\hline
$(3,5,2)$&
$\mathbf{127}$/16s/59091s&
$\mathbf{143}$/20s/160094s&
$\mathbf{143}$/20s/68164s\\\hline
\end{tabular}

\medskip

\begin{tabular}{|c|c|c|c|c|}
\hline
$(m,n,r)$&$s=7$&$s=8$&$s=9$\\
\hline\hline
$(4,4,2)$&
$\mathbf{1074}$/1816s/--&
$\mathbf{818}$/821s/--&
$\mathbf{532}$/349s/--\\\hline
$(3,4,2)$&
$\mathbf{73}$/2.2s/4570s&
$\mathbf{49}$/1.0s/1619s&
$\mathbf{22}$/0.8s/350s\\\hline
$(3,5,2)$&
$\mathbf{143}$/20s/99208s&
$\mathbf{143}$/20s/163532s&
$\mathbf{128}$/18s/263586s\\\hline
\end{tabular}

\medskip

\begin{tabular}{|c|c|c|c|}
\hline
$(m,n,r)$&$s=10$&$s=11$&$s=12$\\
\hline\hline
$(4,4,2)$&
$\mathbf{276}$/92s/--&
$\mathbf{100}$/42s/450988s&
$\mathbf{20}$/1.4s/1970s\\\hline
$(3,4,2)$&
$\mathbf{6}$/0.3s/6.4s&&\\\hline
$(3,5,2)$&
$\mathbf{88}$/13s/67460s&
$\mathbf{40}$/1.9s/4568s&
$\mathbf{10}$/0.8s/114s\\\hline
\end{tabular}

\caption{\label{table:EDsecdet}
Symbolic computations for affine sections of determinantal varieties
with $\Lambda = {\bf 1}$.}
\end{table}

In Table \ref{table:EDsecdet}, we report on some Gr\"obner basis
computations with the \texttt{maple} package \texttt{FGb} for $\Lambda
= {\bf 1}$. Here we used the formulation (\ref{eq:EDdetgeneral}).  The
ED degree, given in bold face, is followed by 
the time, measured in seconds, for computing the graded reverse lexicographic
Gr\"obner basis.  The first timing is obtained by 
performing the computation over the finite field ${\sf
    GF}(65521)$; the second one is obtained by computing over the field of rationals $\Q$. The symbol ``$-$'' means that we did not obtain the Gr\"obner basis after seven days of computation.

An important observation in Table \ref{table:EDsecdet} is the correlation
between the reported running times and the values of ${\rm EDdegree}_{\mathbf 1}$.
The former tell us how many arithmetic operations are needed to find a Gr\"obner basis.
 This suggests that  the ED degree is an accurate measure for
the complexity of solving low-rank approximation problems with symbolic algorithms,
and it serves as a key motivation for computing
ED degrees using advanced tools from algebraic geometry.
This will be carried out in the next section, both for
$\Lambda$ generic and for~$\Lambda = {\bf 1}$.
In particular, we shall arrive at theoretical explanations for the
ED degrees in Tables \ref{table:EDdeterminant}
and \ref{table:EDsecdet}.

\section{Algebraic Geometry}\label{sec:alggeo}

The study of ED degrees for
algebraic varieties was started in \cite{DHOST}.
This section builds on and further develops the geometric  theory in that paper.
We focus on the low rank approximation problem  (\ref{eq:frobnorm2}), and
 we derive general formulas for the ED degrees in
Tables \ref{table:EDdeterminant} and \ref{table:EDsecdet}.

We recall that an affine variety $X\subset\C^{N+1}$ is an {\em affine cone}
if $x\in X$ implies $tx\in X$ for every $t\in\C$. The variety of $m
\times n$-matrices of rank $\leq r$ is an affine cone.  If
$X\subset\C^{N+1}$ is an affine cone, then the corresponding
projective variety $\PP X\subset\PP^{N}$ is well defined.  The ED degree of $\PP X$ is
the ED degree of its affine cone $X$.
The following proposition explains the shift between the third and fourth 
column of Table~\ref{table:EDdeterminant}. More generally, it shows that we can restrict the analysis to linear sections, since the ED degree (for generic weights) 
in the affine case can be deduced from the linear case.

\begin{proposition}\label{prop:affinelinear}
Let $X\subset\C^{N+1}$ be an affine cone,
let ${\mathcal A}^s$ (resp. ${\mathcal L}^{s}$) be a generic affine (resp. linear) subspace 
of codimension $s \geq 1$ in $\C^{N+1}$. Then
\begin{equation}
\label{eq:confirm}
\mathrm{EDdegree}_{\rm gen}(X\cap {\mathcal A}^{s})
\,\,=\,\,
\mathrm{EDdegree}_{\rm gen}(X\cap {\mathcal L}^{s-1}).
\end{equation}
\end{proposition}

\begin{proof}
Let $\overline{X}\subset\PP^{N+1}$ be the projective closure of $X$.
From \cite[Theorem 6.11]{DHOST}, we have
$\mathrm{EDdegree}_{\rm gen}(X)=\mathrm{EDdegree}_{\rm gen}(\overline{X})$, since the transversality assumptions
in that result are satisfied for general weights.
From the equality $\overline{X\cap {\mathcal A}^s}=\overline{X}\cap {\mathcal L}^s$, we conclude 
$\mathrm{EDdegree}_{\rm gen}(X\cap {\mathcal A}^s)=\mathrm{EDdegree}_{\rm gen}(\overline{X}\cap {\mathcal L}^s)=
\mathrm{EDdegree}_{\rm gen}(\PP X\cap {\mathcal L}^{s-1})=\mathrm{EDdegree}_{\rm gen}(X\cap {\mathcal L}^{s-1})$. Here,  the second equality follows from
$\PP X=\overline{X}\cap {\mathcal L}^1$.
\end{proof}

Consider a projective variety $X $ embedded in $ \PP^N$ with a generic system of coordinates.
It was shown in \cite[Theorem 5.4]{DHOST} that
${\rm EDdegree}_{\rm gen}(X)$ is  the sum of the degrees of the polar classes
$\delta_i(X)$. Here, $\delta_i(X)$ denotes the degree
of the {\em polar class} of $X$ in dimension $i$, as in \cite{Holme}. Moreover, if
${\mathcal L}^s$ is a generic linear subspace of codimension $s$ in $\PP^N$ then
$\,\delta_i(X\cap {\mathcal L}^s)=\delta_{i+s}(X)\,$ by \cite[Corollary 6.4]{DHOST}.
We call {\em $s$-th sectional ED degree} of $X$ the number ${\rm EDdegree}_{\rm gen}(X\cap {\mathcal L}^s)$. 
We denote by $X^*$ the dual variety of $X$,
as in \cite[\S 5]{DHOST}, and already seen
 in the proof of Proposition \ref{prop:weightedDuality}.

\begin{corollary}\label{thm:sEDpolar}
The $s$-th {\em sectional ED degree} of $X$ is expressed
in terms of polar classes as
\begin{equation}
\label{eq:sumdelta}
 \mathrm{EDdegree}_{\rm gen}(X\cap {\mathcal L}^s)\,\,=\,\sum_{\ell \ge s}\delta_\ell (X) . 
 \end{equation}
If $\,s \leq \mathrm{codim}(X^*)-1\,$ then 
$X$ and $X \cap {\mathcal L}^s$ have the same generic ED degree.
\end{corollary}

\begin{proof}
This follows from results in Sections 5 and 6 in \cite{DHOST}.
In order to compute $\mathrm{ED degree\ }(X\cap {\mathcal L}^s)$ we have to sum
$\delta_i(X)$ for $i\ge s$. However, it is known that
$\delta_i(X)=0$ if $i\le \textrm{codim}(X^*)-2$.
\end{proof}

A special role in \cite{DHOST} is played by the 
{\em isotropic quadric} $ Q = V(x_0^2 + x_1^2 + \cdots + x_N^2)$ in $\PP^N$.
If $X$ is smooth and transversal to $Q$ then
\cite[Theorem 5.8]{DHOST} gives an explicit formula 
for the ED degree in terms of Chern classes of $X$ $c_i(X)$.
A thorough treatment of Chern classes can be found in \cite{Fulton};
the reader interested in the applications in this paper can be referred to
the basics provided in \cite{DHOST}.
By combining  \cite[Theorem 5.8]{DHOST}
 with Corollary \ref{thm:sEDpolar}, we obtain

\begin{theorem} \label{thm:thmA}
Let  $X \subset \PP^N$ be a smooth projective variety of dimension $M$
and assume that $X$ is transversal to the isotropic quadric $Q$.
Then the $s$-th sectional ED degree of $X$ equals
$$  \mathrm{ED degree}_{\rm gen}(X\cap {\mathcal L}^s)\,\,\, = \,\,\,
\sum_{\ell = s}^{M}\sum_{k=\ell }^{M} (-1)^{M-k} \binom{k+1}{\ell+1} \deg(c_{M-k}(X)). $$
\end{theorem}

\begin{proof} The inner sum is the polar class $\delta_i(X)$;
see the proof of \cite[Thm.~5.8]{DHOST}.
\end{proof}

We now apply Theorem  \ref{thm:thmA} to the situation when 
$M = m+n-2$, $N = mn-1$, and
$X = \PP^{m-1} \times \PP^{n-1}$ is the
Segre variety of  $m\times n$ matrices of rank $1$ in $\PP^N$.
The Chern polynomial of the tangent bundle of $X$ in the Chow ring $A^*(X)=\Z[s,t]/\langle s^m,t^n \rangle$
equals $(1+s)^m(1+t)^n$.  By \cite[page 150]{Holme}, this implies
\begin{equation}
\label{eq:delta_l}
\delta_\ell(X) \,\,\, = \,\, \sum_{k=\ell}^{m+n-2} \! (-1)^{m+n-k} \binom{k+1}{\ell+1} V_k ,
\end{equation}
where $V_k = \deg(c_{M-k}(X)) $ is 
the coefficient of $s^{m-1}t^{n-1}$ in the expansion of 
$(1+s)^m(1+t)^n(s+t)^{k}$. Toric geometers
may view $V_k$ as the sum of the normalized volumes of all $k$-dimensional
faces of the polytope $\Delta_{m-1} \times \Delta_{n-1}$; see
\cite[Cor.~5.11]{DHOST}.

The following result explains the ED degrees in the third column in Table  \ref{table:EDdeterminant},
and it allows us to 
determine this column for any desired value of $m$, $n$ and $s$:

\begin{theorem}
Let $m \leq n$ and $\mathcal{L}$ be a generic linear subspace
of codimension $s$ in $\R^{m \times n}$.
For matrices of rank $1$ or corank $1$, the generic ED degree 
is given by
\begin{equation}
\label{eq:EDdegdual}
\begin{matrix}
\,{\rm EDdegree}_{\rm gen}(\mathcal{L}_{\leq 1}) & = & \delta_s(X) + \delta_{s+1}(X) + \cdots + \delta_{m+n-2}(X), \\ 
{\rm EDdegree}_{\rm gen}(\mathcal{L}_{\leq m-1}) & =  & \delta_0(X) + \delta_1(X) + \cdots + \delta_{mn-2-s}(X),
\end{matrix}
\end{equation}
where $\delta_\ell(X)$ may be computed from (\ref{eq:delta_l}).
\end{theorem}

\begin{proof}
The dual in $\PP^{mn-1}$ to the Segre variety $X= \PP^{m-1} \times \PP^{n-1}$
is the variety $X^*$ of matrices of rank $\leq m-1$.
By \cite[Theorem 2.3]{Holme}, we have $\delta_\ell(X)=\delta_{mn-2-\ell}(X^*)$ for all $\ell$.
With this duality of polar classes,
 the result follows from  Corollary \ref{thm:sEDpolar} and \cite[Theorem 5.4]{DHOST}.
\end{proof}

\begin{example} \rm
Fix $m{=}n{=}3$. For matrices of rank $1$,  formulas
(\ref{eq:delta_l}) and (\ref{eq:EDdegdual}) give
$$\begin{array}{c|c|c|c|c|c|c|c|c|}
s = {\rm codim}(\mathcal{L}) &0&1&2&3&4&5&6&7\\
\hline
V_s & 9 & 18 & 24 & 18 & 6 & 0  & 0  & 0  \\
\hline 
\delta_s (X)&3&6&12&12&6&0&0&0\\
\hline
{\rm EDdegree}_{\rm gen}(\mathcal{L}_{\leq 1})  &39&36&30&18&6&0&0&0
\end{array}$$
Duality for polar classes yields the formulas for 
$3 \times 3$-matrices of rank $r=2$ in $\mathcal{L}$:
$$\begin{array}{c|c|c|c|c|c|c|c|c|} 
s = {\rm codim}(\mathcal{L}) &0&1&2&3&4&5&6&7\\
\hline
\delta_s(X^*) = \delta_{7-s}(X) &0&0&0&6&12&12&6&3\\
\hline
{\rm EDdegree}_{\rm gen}(\mathcal{L}_{\leq 2})
&39&39&39&39&33&21&9&3
\end{array}$$
This is our theoretical derivation of the third column
in  Table  \ref{table:EDdeterminant} for $n=3$ and generic 
 $\Lambda$. \hfill $\diamondsuit $
\end{example}

Writing down closed formulas for intermediate values of $r$
is more difficult: it involves some Schubert calculus.
However, ${\rm EDdegree}_{\rm gen}(\mathcal{L}_{\leq r})$
can be conveniently computed with the following script in {\tt Macaulay2} \cite{M2}.
It is a slight generalization of that in~\cite[Example 7.10]{DHOST}:
\begin{verbatim}
loadPackage "Schubert2"
ED=(m,n,r,s)->
(G = flagBundle({r,m-r}); (S,Q) = G.Bundles;
X=projectiveBundle (S^n); (sx,qx)=X.Bundles;
d=dim X; T=tangentBundle X;
sum(toList(s..m*n-2),i->sum(toList(i..d),j->(-1)^(d-j)*
   binomial(j+1,i+1)*integral(chern(d-j,T)*(chern(1,dual(sx)))^(j)))))
\end{verbatim}
The function {\tt ED(m,n,r,s)} computes the ED degree of the variety
of $m\times n$ matrices of rank $\le r$, in general coordinates, cut with 
a generic linear space
of codimension $s$ in  $\PP^{mn-1}$.
 For $s=0$ this is precisely the function displayed in
\cite[Example 7.10]{DHOST}.

\begin{example} \label{ex:1350}
 \rm
The bold face ED degrees in Table \ref{table:EDsecdet} were computed 
for unit weights $\Lambda = {\bf 1}$. To find the analogous numbers
for generic weights $\Lambda$, we run our {\tt Macaulay2} code as follows:
\begin{verbatim}
apply(12,s->ED(4,4,2,s))
    {1350, 1350, 1350, 1350, 1330, 1250, 1074, 818, 532, 276, 100, 20}
apply(12,s->ED(3,4,2,s))
    {83, 83, 83, 83, 83, 83, 73, 49, 22, 6, 0, 0}
apply(12,s->ED(3,5,2,s))
    {143, 143, 143, 143, 143, 143, 143, 143, 128, 88, 40, 10}  
\end{verbatim}   
\end{example}

At this point, we wish to reiterate the main thesis of this paper, namely 
that knowing the ED degree ahead of time is 
useful for practitioners who seek to find and certify
 the global minimum in  the optimization problem (\ref{eq:frobnorm2}),
 and to bound the number of local minima.
 The following example illustrates this for one 
   of the numbers {\tt 83}  in
the output in Example \ref{ex:1350}.

\begin{example} \rm  We here solve the
generic weighted structured low-rank approximation problem over the reals
with parameters $m=3$, $n=4$, $r=2$ and $s=2$.
Consider  the instance
$$\begin{array}{rl}
& U=\begin{bmatrix}-9&    4&    9 &   -10\\
                10  &  6  &  1 &    -9\\
                10  &  5  &  7   &   6\end{bmatrix}\qquad\qquad 
\Lambda=\begin{bmatrix}
8  &  6 &   8 &   2\\
1  &  8  &  7  &  9\\
7  &  2 &   4  &  6\end{bmatrix}\medskip
\\
L_1(X) =&  -10 x_{11} + 4 x_{12} + 6 x_{13} + 8 x_{14} + 4 x_{21} - 9 x_{22}+\\
& x_{23} - 10 x_{31}
    {-} 10 x_{32} {-} 8 x_{33} {+} 2 x_{34} {-} 1,\medskip\\
L_2(X) = & 2 x_{11} + 7 x_{12} + 3 x_{13} - 7 x_{14}
     - 4 x_{21} - 6 x_{22} - 7 x_{23} + \\&  5 x_{24} + 8 x_{31} +2 x_{33} + 3 x_{34} - 1.
   \end{array}$$
We wish to find the matrix $X$ of rank at most $2$ that satisfies the 
affine constraints $L_1(X)=L_2(X)=0$ and is 
 nearest to $U$.
       Using Gr\"obner
   bases computations and real isolation techniques via the {\tt
     Maple} packages {\tt FGb} and {\tt fgbrs}, we find that the
   weighted distance function has 83 complex critical points. This
     matches the theoretical value $\mathrm{ED}(3,4,2,2)=83$ provided in Example~\ref{ex:1350}, so that we are guaranteed that there are no further critical points. Among them, seven are real and we obtain certified
   numerical approximations of their values:
   
 $$   \footnotesize
\begin{bmatrix}
       \phantom{-}0.764&\!\!     -1.457 \! &\!\!      \phantom{-}2.436&\!\!     \phantom{-}1.870\\
       \phantom{-}0.753&\!\!    -0.0154 \! &\!\!    \phantom{-}0.030&\!\!    -7.437\\
       \phantom{-}2.020&\!\!      -4.371\! &\!\!      \phantom{-}7.308&\!\!    \phantom{-}8.330
\end{bmatrix} \!
\begin{bmatrix}
         -8.0341 \! &\!\!    \phantom{-}4.127&\!\!    \phantom{-}9.055&\!\!     \phantom{-}5.364\\
         \phantom{-}16.936 \! &\!\!     \phantom{-}2.930&\!\!    -1.330&\!\!    -4.220\\
         \phantom{-}9.429 \! &\!\!     \phantom{-}7.525&\!\!    \phantom{-}8.258&\!\!     \phantom{-}1.242
\end{bmatrix} \!$$
$$\footnotesize \begin{bmatrix}
         -8.215 \! &\!\!    \phantom{-}5.033&\!\!    \phantom{-}9.965&\!\!     \phantom{-}1.647\\
         \phantom{-}16.848 \! &\!\!     \phantom{-}4.259&\!\!    \phantom{-}0.423&\!\!    -3.669\\
         \phantom{-}9.070 \! &\!\!     \phantom{-}6.218&\!\!    \phantom{-}5.842&\!\!     -2.054
\end{bmatrix}
 \begin{bmatrix}
         -8.586&\!\!    -1.743&\!\!    \phantom{-}1.591&\!\!     \phantom{-}2.436\\
         \phantom{-}11.191&\!\!     \phantom{-}2.985&\!\!     -4.232&\!\!    -7.159\\
         \phantom{-}10.351&\!\!     \phantom{-}0.292&\!\!    \phantom{-}3.567&\!\!     \phantom{-}7.185
\end{bmatrix}$$
$$\footnotesize \begin{bmatrix}
        -4.853&\!\!    \phantom{-}4.081&\!\!      \phantom{-}6.301&\!\!     -6.349\\
        -6.067&\!\!    \phantom{-}5.029&\!\!      \phantom{-}8.600&\!\!     -8.251\\
        \phantom{-}2.616&\!\!     -2.455&\!\!    -0.878&\!\!    \phantom{-}2.327
\end{bmatrix}
\begin{bmatrix}
        -2.308&\!\!     -4.584&\!\!    \phantom{-}3.566&\!\!     -5.484\\
        -0.205&\!\!    -2.210&\!\!    \phantom{-}0.668&\!\!    -3.178\\
        -2.276&\!\!     \phantom{-}0.983&\!\!    \phantom{-}2.444&\!\!     \phantom{-}2.810
\end{bmatrix}$$
$$  \footnotesize
\begin{bmatrix}
         -9.664&\!\!    \phantom{-}2.805&\!\!     \phantom{-}7.113&\!\!     -10.754\\
         \phantom{-}14.942&\!\!     \phantom{-}6.520&\!\!     \phantom{-}3.149&\!\!     -8.783\\
         \phantom{-}8.344&\!\!     \phantom{-}0.615&\!\!    -2.185&\!\!    \phantom{-}2.177
 \end{bmatrix}
$$

The last matrix is the closest critical point on the manifold of rank
$2$ matrices satisfying $L_1 = L_2 = 0$.
 This computation takes 1002 seconds and the most
time-consuming step is the computation of the Gr\"obner basis. In
order to certify that the global minimum is among these matrices, we
also solve the same low-rank approximation problem for rank $1$
matrices. Using the same method, this provides us with $11$ rank $1$
matrices with real entries in $79$ seconds. None of them is closer to
$U$ than the best rank $2$ approximation. Consequently, the global
minimum of the weighted distance is reached at the last matrix in the above list.

For comparison
purposes, with the same constraints $L_1,L_2$ and  same data
matrix $U$ but by taking the Frobenius distance (\emph{i.e.} $\Lambda$
is the unit matrix), the number of complex critical points is $43$. Five
of them are real. Here, it takes only 27 seconds to find the global minimizer.
   These computations have been performed on an {\tt Intel Xeon E7540/2.00GHz}.
    \hfill $\diamondsuit $
\end{example}

In Table \ref{table:EDdeterminant} and Example \ref{ex:1350}
we observed that
the sectional ED degree for generic $\Lambda$ does not
depend on $s = {\rm codim}(\mathcal{L})$, provided $s$ is small.
The following corollary explains this.

\begin{corollary}\label{cor:EDsequal2}
For a generic  linear subspace $\mathcal{L} $
of codimension $s < r(r+n-m)$,
$$ {\rm EDdegree}_{\rm gen}(\mathcal{L}_{\le r})
\,\,=\,\, {\rm EDdegree}_{\rm gen}(\C^{m \times n}_{\le r}). $$
\end{corollary}

\begin{proof}
Let $X$ be the variety of matrices of rank $\leq r$.
Its dual $X^*$ is the variety of matrices of rank $\leq m-r$ and has codimension ${\rm codim}(X^*) = (r+n-m)r$. This implies
$\delta_\ell(X) = 0$ for $\ell < (r+n-m)r-1$.
The assertion  follows  from Corollary
\ref{thm:sEDpolar}.~\end{proof}

Corollary  \ref{cor:EDsequal2} can be stated informally like this:
in the setting of generic weights and generic linear spaces of matrices with sufficiently 
high dimension,
the algebraic complexity of structured low-rank approximation
agrees with that of ordinary low-rank approximation.

\smallskip

Shifting gears, we now consider the case of
unit weights $\Lambda = {\bf 1}$. Thus, we fix
$Q = V(\sum x_{ij}^2)$ as the isotropic
quadric in $\PP^{mn-1}$.
Let $X= \PP^{m-1} \times \PP^{n-1}$ denote
the Segre variety of $m\times n$ matrices of rank~$1$ in $\PP^{mn-1}$,
and let $Z = {\rm Sing}(X \cap Q)$ denote the
non-transversal locus of the intersection of $X$ with $Q$.
The  dual variety $X^*$  consists of all matrices of rank $\leq m-1$ in $\PP^{mn-1}$.
We  conjecture that the following formula (put $m=n$)  holds for
 the gap between the third and the first  column of Table \ref{table:EDdeterminant},
(or between the fourth and the second, as well), 
\begin{equation}
\label{eq:3ED}
 {\rm EDdegree}_{\rm gen}(X^* \cap {\mathcal L}^s)-{\rm EDdegree}_{\bf 1}(X^* \cap {\mathcal L}^s) \,\,= \,\,
 {\rm EDdegree}_{\rm gen}(Z \cap {\mathcal L}^s ). 
\end{equation}

To compute the right-hand side, and to test this conjecture, we use 

\begin{lemma} \label{lemma:nontrans}
The locus where $Q$ meets $X=\PP^{m-1}\times\PP^{n-1}$
 non-transversally in $\PP^{mn-1}$ is the product $\,Z=Q_{m-2}\times Q_{n-2}$,
where $Q_{i-2}$ denotes a general quadratic hypersurface in $\PP^{i-1}$.
\end{lemma}
\begin{proof}
The Segre variety
$X$ meets $Q$ in the union of two irreducible components, $\PP^{m-1}\times Q_{n-2}$
and $Q_{m-2}\times \PP^{n-1}$. The non-transversality locus is the intersection of these components.
\end{proof}

\begin{example} \rm 
Let $m=n=2$, so $X$ and $X^*$ represent $3 \times 3$-matrices of rank $1$ and rank $\le 2$  respectively.
Here $Z=Q_{1}\times Q_{1}$ corresponds to  the Segre quadric $\PP^1\times\PP^1$,
embedded in $\PP^8$ with the line bundle $\mathcal{O}(2,2)$. This is a toric surface
whose polygon $P$ is twice a regular square.
The facial volumes as in \cite[Corollary 5.1]{DHOST} are
$V_0 = 4$, $V_1 = 8 $ and $V_2 = 8$, and  hence 
$$ \delta_0(Z)=4-2\cdot 8+3\cdot 8=12 \,, \,\quad
\delta_1(Z)=-8+3\cdot 8=16 \, , \,\quad
\delta_2(Z)=8 . $$
We fill this into a table and, using Corollary \ref{thm:sEDpolar},
 we compute the sectional ED degree:
$$\begin{array}{c|c|c|c|c|c|c|c|c|}s&0&1&2&3&4&5&6&7\\
\hline
\delta_s(Z)&12&16&8&0& 0 &  0 & 0 &  0 \\
\hline
\mathrm{ED degree}_{\rm gen}(Z\cap {\mathcal L}^s)&36&24&8&0&0&0&0&0\\
\hline
\mathrm{ED degree}_{\rm gen}(X^*\cap {\mathcal L}^s)&39&39&39&39&33&21&9&3\\
\hline
\mathrm{ED degree}_{\bf 1}(X^*\cap {\mathcal L}^s) &3&15&31&39&33&21&9&3\\
\end{array}$$
The last two lines are taken from 
Table~\ref{table:EDdeterminant}, and they confirm the formula (\ref{eq:3ED}).
 \hfill $\diamondsuit$ 
\end{example}




Combining Lemma \ref{lemma:nontrans}, Corollary \ref{thm:sEDpolar} and the proof
 of \cite[Theorem 5.8]{DHOST}, and abbreviating
$W_{j}=\deg(c_{m+n-4-j}(Q_{m-2}\times Q_{n-2}))$,
the right-hand side of (\ref{eq:3ED}) 
can be expressed~as
\begin{equation} \label{discrepancy} \begin{matrix} 
\displaystyle \sum_{i=s}^{m+n-4}\,\,\, \sum_{j=i}^{m+ n-4}(-1)^{m+n-4-j}\binom{j+1}{i+1}W_{j}.
\end{matrix}
\end{equation}  
Moreover,  $W_{j}$ is equal to the coefficient of
$t^{m-2}s^{n-2}$ in the rational generating function
$$ 4\frac{(1+t)^m(1+s)^n}{(1+2t)(1+2s)}(t+s)^{j}. $$
This computation
allows us to
extend Table~\ref{table:EDdeterminant} to any desired value of $m$, $n$ and $s$.

Changing topics,
we now  consider the 
case when  $\mathcal{L}$ is the space of
{\em Hankel matrices}.
The computation of low-rank approximation of Hankel matrices 
will be our topic in Section~\ref{sec:HankelSylvester}, where
we focus on  algebraic geometry and
formulas for  generic ED degree.

Set $d = p{+}q{-}2$ and 
let $X_{d,r}$ denote the variety of
$p \times q$ Hankel matrices of rank $\leq r$.
See (\ref{ex:hankel56}) for examples.
This variety lives in the
projective space $\PP^d = \PP(S^d \C^2)$, whose points
represent binary forms of degree $d$. 
Thus $X_{d,1}$ is the rational normal curve of degree $d$,
and $X_{d,r}$ is the $r$th secant variety of this curve.
We have  $\,{\rm dim}(X_{d,r}) = 2r-1$ for $r+1\le\min(p,q)$.

\begin{theorem} \label{thm:EDsecant}
Let $d=p+q-2$ and $r+1\le\min(p,q)$. The generic ED degree of the variety $X_{d,r}$
of $p \times q$ Hankel matrices of rank $\leq r$
 in $\PP^d$ equals
\begin{equation}
\label{eq:2binomials}
 {\rm EDdegree}_{\rm gen}(X_{d,r})\,\,\, = \,\,\,\sum_{i=0}^r \binom{d+1-r}{i} \binom{d-r-i}{r-i} 2^{r-i}. 
 \end{equation}
\end{theorem}

\begin{proof}
The sum in (\ref{eq:2binomials})
is the coefficient of $z^r$~in the generating function
\begin{equation}
\label{eq:genfunc}
\frac{(1+z)^{d+1-r}}{(1-2z)^{d-2r+1}} .
\end{equation}
 The conormal variety  of $X_{d,r}$ is the closure~$\mathcal N_{X_{d,r}}$ of
 the set
$$ \bigl\{\,(f,g)\,|\,\,\mathrm{rank}(f) =   r\,\,
{\rm and} \,\,g\textrm{\ is tangent  to\ }
X_{d,r}\textrm{\ at\ }f \bigr\} \,
\subset \,
\PP(S^d \C^2 )\times\PP(S^d(\C^2)^{*}). $$
The homology class of $\mathcal N_{X_{d,r}}$ is given by a binary form.
We will show that the sum  $\sum_i \delta_i(X_{d,r})$
of its coefficients is the asserted coefficient of (\ref{eq:genfunc}).
By \cite[(5.3)]{DHOST}, this proves the~claim.

Let $p_1$, $p_2$ be the two projections.
The images of the conormal variety $\mathcal N_{X_{d,r}}$ are
$$ p_1(\mathcal N_{X_{d,r}})\,=\,X_{d,r}\quad 
\hbox{and} \quad p_1(\mathcal N_{X_{d,r}})\,=\,X_{d,r}^*.$$
We desingularize $X_{d,r}$ by considering
$\mathrm{Sym}^r(\PP^1)\simeq\PP^r$. The desingularization map
is given by the scheme-theoretic intersection of  
the rational normal curve of degree $r$ with a  hyperplane.
A point in $\PP^r$, identified with a hyperplane,  gives $r$ points on $X_{d,1} \simeq \PP^1$.
 Their linear span in $\PP^d$
defines a rank $r$ bundle on $\PP^r$,
known as the {\em Schwarzenberger bundle} \cite[\S 6]{DK}.
This is the kernel of the bundle map $\sO^{d+1}\to\sO(1)^{d+1-r}$.
In the same way, we desingularize the conormal variety
$\mathcal N_{X_{d,r}}$ by the fiber product over $\PP^r$ of the
 projectivization of the 
Schwarzenberger bundle $E_{d,r}={\rm kernel}(\sO^{d+1}\to\sO(1)^{d+1-r})$
and of the projective bundle of $\sO(2)^{d-2r+1}$.
Exactly as in the proof of \cite[Proposition 4.1]{BR}, the degrees of the polar classes
 of $X_{d,r}$ are 
 $$\delta_{r+i-1}(X_{d,r})\,\,=\,\int_{\PP^r}s_{i}(E_{d,r})s_{r-i}(\sO(2)^{d-2r+1}). $$
The total Segre class of $E_{d,r}$ is $(1+z)^{d+1-r}$.
The total Segre class of $\sO(2)^{d-2r+1}$ is
$\frac{1}{(1-2z)^{d-2r+1}}$.
By multiplying them we obtain the degree sum
 of the polar classes, thus proving (\ref{eq:genfunc}).
\end{proof}

\begin{corollary} \label{cor:powerof3}
The generic ED degree of the
hypersurface  $X_{2r,r}$  defined by the
 Hankel determinant
of format  $(r+1)\times (r+1)$ is equal to
\begin{equation}
\frac{3^{r+1}-1}{2} \quad = \quad
\hbox{the coefficient of $z^r$ in} \,\,\,
\frac{\,(1+z)^{r+1}}{1-2z} .
\end{equation}
\end{corollary}

This corollary means that the ED degree of the 
$(r{+}1) \times (r{+}1)$  Hankel determinant
agrees with the ED degree of the 
 general symmetric $(r{+}1) \times (r{+}1)$ determinant.
By ED duality \cite[Theorem 5.2]{DHOST}, this also
the ED degree of the second Veronese  embedding of $\PP^r$;
see  \cite[Example 5.6]{DHOST}.
If we consider Hankel matrices of fixed rank $r$ then
we obtain polynomiality:

\begin{corollary} \label{cor:forfixed}
For fixed $r$, the generic ED degree of $X_{d,r}$ is a polynomial
of degree $r$ in $d$.
\end{corollary}

For example, we find the following explicit polynomials 
when the rank $r$ is small:
$$ \begin{matrix}
 {\rm EDdegree}_{\rm gen}(X_{d,1}) & = & 3d-2 ,\\
 {\rm EDdegree}_{\rm gen}(X_{d,2}) & = &  (9d^2-39d+38)/2 ,\\
 {\rm EDdegree}_{\rm gen}(X_{d,3}) & = & ( 9 d^3-99 d^2+348d-388)/ 2 ,\\
 {\rm EDdegree}_{\rm gen}(X_{d,4}) & = &  ( 27d^4-558d^3+4221d^2-13818d+16472)/ 8.
  \end{matrix}
 $$
 The values of these polynomials are the entries in
 the left columns  in Table~\ref{table:EDHankel} below.

\section{Hankel and Sylvester Matrices}\label{sec:HankelSylvester}

In this section we study the weighted low-rank approximation 
problem for matrices with a special structure that is given by
equating some matrix entries and setting others to zero.
One such family consists of the Hurwitz matrices in \cite[Theorem 3.6]{DHOST}.
We here discuss Hankel matrices, then 
catalecticants, and finally 
 Sylvester matrices.
The corresponding applications are
low-rank approximation of symmetric tensors
and approximate greatest common divisors.

The {\em Hankel matrix} $H[p,q]$ of format $p \times q$ has
the entry $x_{i+j-1}$ in row $i$ and column $j$.
So, the total number of unknowns is $n = p+q-1$.
We are most interested in the case when
this matrix is square or almost square.
The {\em Hankel matrix of order $n$}
is  $H[(n{+}1)/2,(n{+}1)/2]$ if $n$ is odd,
and it is  $H[(n/2,(n{+}2)/2]$ if $n$ is even.
We denote this matrix by $H_n$. For instance,
\begin{equation}
\label{ex:hankel56}
 H_5  = \begin{bmatrix}
 x_1 & x_2 & x_3 \\
 x_2 & x_3 & x_4 \\
 x_3 & x_4 & x_5
 \end{bmatrix} \quad \hbox{and} \quad
 H_6  = \begin{bmatrix}
 x_1 & x_2 & x_3 & x_4 \\
 x_2 & x_3 & x_4 & x_5 \\
 x_3 & x_4 & x_5 & x_6
 \end{bmatrix} . 
 \end{equation}
 For approximations by low-rank Hankel matrices, 
we consider  three natural weights:
\begin{itemize}
\item the matrix $\Omega_n$ has entry
$\,1/{\rm min}(i{+}j{-}1,n{-}i{-}j{+}2 )\, $ in row $i$ and column $j$;
\item the matrix ${\bf 1}_n$ has all entries equal to $1$;
\item the matrix $\Theta_n$ has 
$\binom{n-1}{i+j-2}/{\rm min}(i{+}j{-}1,n{-}i{-}j{+}2 ) $
in row $i$ and column~$j$.
\end{itemize}
We encountered these matrices for $n=5$ in 
 Example \ref{eq:hankel33}. For $n = 6$ we have
$$ \Omega_6 \, = \,
\begin{bmatrix} 
   1 & 1/2 & 1/3 & 1/3 \\
1/2 & 1/3  & 1/3 & 1/2 \\
1/3 & 1/3 & 1/2 & 1 
\end{bmatrix} \quad \hbox{and} \quad
 \Theta_6 \, = \,
\begin{bmatrix} 
   1 & 5/2 & 10/3 & 10/3 \\
5/2 & 10/3  & 10/3 & 5/2 \\
10/3 & 10/3 & 5/2 & 1 
\end{bmatrix} .
$$
The weights $\Omega_n$ represent
the usual Euclidean distance in $\R^n$,
the unit weights ${\bf 1}_n$ give the
Frobenius distance in the ambient matrix space,
and the weights $\Theta_n$ give
the natural metric in the space
of symmetric $2 {\times} 2 {\times} \cdots {\times} 2$-tensors.
Such a tensor corresponds to a binary form
$$ F(s,t) \,= \,
\sum_{i=1}^n \binom{n-1}{i-1} \cdot x_i \cdot s^{n-i} \cdot t^{i-1}. $$
The Hankel matrix $H_n$ has rank $1$ if and only if
$F(s,t)$ is the $(n{-}1)$st power of a linear form. More generally,
if $F(s,t)$ is the sum of $r$ powers of linear forms
then $H_n$ has rank $\leq r$.
As we saw in \S 3, this locus corresponds to the
$r$th secant variety of the rational normal curve in $\PP^{n-1}$.
Various ED degrees for our three
weight matrices are displayed in Table \ref{table:EDHankel}.

\begin{table}
\centering
\begin{tabular}{|c|c|c|c|c|}
\hline
\multicolumn{5}{|c|}{$\Lambda = \Omega_n $}\\
\hline
$\! n \backslash r \! $ &$1$&$2$&$3$&$4$\\
\hline\hline
$3$& 4  & & & \\
$4$& 7 & & & \\
$5$& 10 & 13 & & \\
$6$& 13 & 34 & & \\
$7$& 16 & 64 & 40 &\\
$8$& 19 & \!\!103\!\! & \!\!142\!\! &\\
$9$& 22 & \!\!151 \!\! & \!\!334\!\!  & \!\!121\!\! \\
\hline
\end{tabular} \quad
\begin{tabular}{|c|c|c|c|c|}
\hline
\multicolumn{5}{|c|}{$\Lambda = {\bf 1}_n $}\\
\hline
$ \! n \backslash r \! $ &$1$&$2$&$3$&$4$\\
\hline\hline
$3$&2 & & & \\
$4$&7 & & & \\
$5$&6 &9 & & \\
$6$&13 &34 & & \\
$7$&10 &38&34&\\
$8$&19 & \!\!103\!\!& \!\!142\!\!&\\
$9$&14 & \!\!103 \!\!&\!\! 246\!\!&\!\!113\!\!\\
\hline
\end{tabular} \quad
\begin{tabular}{|c|c|c|c|c|}
\hline
\multicolumn{5}{|c|}{$\Lambda = \Theta_n $}\\
\hline
$\! n \backslash r \! $ &$1$&$2$&$3$&$4$\\
\hline\hline
$3$& 2& & & \\
$4$& 3& & & \\
$5$& 4&7
& & \\
$6$& 5& 16 & & \\
$7$& 6& 28& 20&\\
$8$& 7& 43& 62&\\
$9$& 8& 61&\!\! 134\!\! &53\\
\hline
\end{tabular}
\caption{\label{table:EDHankel}
Weighted ED degrees for Hankel matrices of order $n$ and rank $r$.}
\end{table}

The entries in the leftmost chart in Table \ref{table:EDHankel} come from 
Theorem~\ref{thm:EDsecant}. Indeed, the variety of Hankel matrices $H_n$ of rank $\leq r$ is
precisely the secant variety $X_{n-1,r}$ we discussed  in Section 3.
The weight matrix $\Lambda  = \Omega_n$ exhibits the generic ED degree for that variety.
The columns on the left of Table \ref{table:EDHankel} are
the values of the polynomials in Corollary \ref{cor:forfixed},
and the diagonal entries $4, 13, 40, 121,\ldots$ are given by
Corollary \ref{cor:powerof3}.

All ED degrees in Table \ref{table:EDHankel}
were verified using Gr\"obner basis computations
over $\mathsf{GF}(65521)$ using the {\tt maple} package {\tt FGb} \cite{Fau02}.
The running times are closely tied to the valued of the ED degrees,
and they are similar to those reported in Table \ref{table:EDsecdet}.
Gr\"obner bases over $\mathbb{Q}$ can also be computed fairly easily
whenever the ED degree is below $100$, and for those cases we can  locate 
all real critical points using {\tt fgbrs}. However, for larger
instances, exact symbolic solving over $\mathbb{Q}$
becomes a considerable challenge due to the growth in coefficient size.

 Hankel matrices of rank $r$ correspond to 
symmetric $2 {\times} 2 {\times}  \cdots {\times} 2$-tensors
of tensor rank~$r$, and these can be represented by binary forms 
that are sums of $r$ powers of linear forms.
That is the point of the geometric discussion in Section \ref{sec:alggeo}.
This interpretation extends to symmetric tensors of arbitrary format,
with the rational normal curve replaced with the Veronese variety.
For a general study of
low-rank approximation of symmetric tensors
see Friedland and Stawiska \cite{FS}.
In general, there is no straightforward representation
of low rank tensors by low rank matrices with special structure.
However, there are some exceptions,  notably for rank $r=2$ 
tensors, by the results of Raicu \cite{Rai} and others in the 
recent tensor literature. We refer to Landsberg's book
  \cite{Lan}, especially Chapters 3, 7 and 10.
The resulting generalized Hankel matrices are  known as
{\em catalecticants} in the commutative algebra literature,
or as {\em moment matrices} in the optimization literature.
We now present a case study that arose from
a particular application in biomedical imaging.

We consider the following catalecticant matrix of format $6 \times 6$:
 $$ X \quad = \quad \begin{bmatrix}x_{400}& x_{310}& x_{301}& x_{220}& x_{211}& x_{202}\\
 x_{310}& x_{220}& x_{211}& x_{130}& x_{121}& x_{112}\\
 x_{301}& x_{211}& x_{202}& x_{121}& x_{112}& x_{103}\\
 x_{220}& x_{130}& x_{121}& x_{040}& x_{031}& x_{022}\\
 x_{211}& x_{121}& x_{112}& x_{031}& x_{022}& x_{013}\\
 x_{202}& x_{112}& x_{103}& x_{022}& x_{013}& x_{004}
\end{bmatrix}$$
The fifteen unknown entries are the coefficients of a ternary quartic
$$ \begin{matrix}
F(s,t,u) =
x_{400} s^4 {+} 
x_{040} t^4 {+} 
x_{004} u^4 + 
6 x_{220} s^2 t^2  {+}
6 x_{202} s^2 u^2 {+}
6 x_{022} t^2 u^2 \\ \qquad \quad +\, 
4 x_{310} s^3 t +
4 x_{301} s^3 u {+}
4 x_{130} s t^3 {+} 
4 x_{031} t^3 u {+} 
4 x_{103} s u^3 \\ \qquad\!\! +\,
4 x_{013} t u^3 +
12 x_{211} s^2 t u {+} 
12 x_{121} s t^2 u {+}
12 x_{112} s t u^2. \end{matrix} $$
The table $(x_{ijk})$ can be regarded as a 
symmetric tensor of format $3 \times 3 \times 3 \times 3$.
The coefficients  in $F(s,t,u)$ indicate
the multiplicity with which the $15$ unknowns occur
among the $3^4 = 81$ coordinates of that tensor.
To model the invariant metric in the tensor space
$\R^{3 \times 3 \times 3 \times 3}$ in our matrix representation, we use the weight matrix
$$ \Theta \,= \,
\begin{bmatrix}
 1 & 2 & 2 & 2  & 3 & 2 \\
 2 & 2 & 3 & 2  & 3 & 3 \\
 2 & 3 & 2 & 3 & 3 & 2 \\
 2 & 2 & 3 & 1 & 2 & 2 \\
 3 & 3 & 3 & 2 & 2 & 2 \\
 2 & 3 & 2 & 2 & 2 & 1 
 \end{bmatrix} .
$$
The problem is to approximate a given catalecticant matrix
$U = (u_{ijk})$
by a rank $2$ matrix with respect to  $\Theta$.
The expected number of critical points is as follows.

\begin{proposition} \label{prop:195}
Let $\mathcal{L}$ be the $15$-dimensional subspace of 
catalecticants $X$ in
$\R^{6 \times 6}$.
Then 
$${\rm EDdegree}_\Theta(\mathcal{L}_{\leq 2}) =  195
\quad \hbox{and} \quad
{\rm EDdegree}_{\rm gen}(\mathcal{L}_{\leq 2}) = 1813.
$$
\end{proposition}

The proof is a  computation as explained below. We first discuss an application.

\begin{example} \label{ex:schultz1} \rm
We consider the following symmetric $3 \times 3 \times 3 \times 3$-tensor:
$$
\begin{array}{|c|c|}
\hline
u_{400}&0.1023\\
u_{220}&0.0039\\
u_{310}& -0.002\\
u_{103}&0.0196\\
u_{211}&-0.00032569\\
\hline
\end{array}\quad\quad
\begin{array}{|c|c|}
\hline
u_{040}&0.0197\\
u_{202}&0.0407\\
u_{301}&0.0581\\
u_{031}&0.0029\\
u_{121}&-0.0012\\
\hline
\end{array}\quad\quad
\begin{array}{|c|c|}
\hline
u_{004}&0.1869\\
u_{022}&-0.00017418\\
u_{130}&0.0107\\
u_{013}&-0.0021\\
u_{112}&-0.0011\\
\hline
\end{array}
$$
This tensor was given to us by Thomas Schultz, who heads the
Visualization and Medical Image Analysis Group at the University of Bonn.
It represents a fiber distribution function,
estimated from diffusion Magnetic Resonance Imaging.
See \cite{Sch} for more information.
 \hfill $\diamondsuit$
\end{example}

We present an algebraic formulation of our
problem which was found to be suitable for symbolic computation.
Introducing six unknowns $a,b,c,d,e,f$, we 
parametrize the $6$-dimensional variety of
symmetric $3 \times 3 \times 3 \times 3$-tensors of rank $2$
by the ternary quartics
$$ \tilde F(s,t,u) \quad = \quad a \cdot (s+b t +c u)^4\,+\,d \cdot (s+e t+f u)^4. $$ 
Just like in the discussion in Remark \ref{differentpara:rmk}
and after Proposition \ref{prop:generalformulation}, the image of this parametrization 
is a dense open subset of the symmetric $3 \times 3 \times 3 \times 3$-tensors of rank~$2$. 
Covering all rank $2$ tensors can be achieved with 
three parametrizations as above.

Written out explicitly, this parametrization takes the form
$$
\begin{array}{rcl}
x_{400} &\!\!=\!\!& a+d \\
x_{220}&\!\!=\!\!& a b^2+d e^2\\ 
x_{310} &\!\!=\!\!& a b+d e \\
 x_{103} &\!\!=\!\!& a c^3+d f^3\\ 
x_{211} &\!\!=\!\!& a b c+d e f \\ 
\end{array}\quad\quad
\begin{array}{rcl}
x_{040} &\!\!=\!\!& a b^4+d e^4\\
x_{202} &\!\!=\!\!& a c^2+d f^2 \\ 
x_{301} &\!\!=\!\!& a c+d f\\ 
x_{031} &\!\!=\!\!& a b^3 c+d e^3 f  \\ 
x_{121}&\!\!=\!\!& a b^2 c+d e^2 f \\
\end{array}\quad\quad
\begin{array}{rcl}
x_{004} &\!\!=\!\!& a c^4+d f^4\\ 
x_{022} &\!\!=\!\!& a b^2 c^2+d e^2 f^2\\ 
x_{130} &\!\!=\!\!& a b^3+d e^3 \\ 
x_{013} &\!\!=\!\!& a b c^3+d e f^3\\ 
x_{112} &\!\!=\!\!& a b c^2+d e f^2 \\ 
\end{array}\quad\quad
$$

Note that our parametrization is $2$ to $1$: every rank $2$ catalecticant $X$
has two preimages, which are related by swapping the vectors
$(a,b,c)$ and $(d,e,f)$.
The fiber jumps in dimension over the singular locus, 
which consists of matrices $X$ of rank $1$. Their
preimage in parameter space is given by the ideal
 $\langle a d\rangle\cap \langle b-e, c-f\rangle$. 
The chosen weight matrix $\Theta$ now specifies the following
   unconstrained optimization problem. We seek to find the minimum in $\R^6$ of
 $$\begin{array}{@{}c@{}}
 G(a,b,c,d,e,f) \,\,= \,\,
 (u_{400}-a-d)^2+(u_{040}-a b^4-d e^4)^2+(u_{004}-a c^4-d f^4)^2 \qquad \qquad \\
+ 6 (u_{220}-a b^2-d e^2)^2+ 6 (u_{202}-a c^2-d f^2)^2
+6 (u_{022}-a b^2 c^2-d e^2 f^2)^2 \\
 +4 (u_{310}-a b-d e)^2+4 (u_{301}-a c-d f)^2+4 (u_{130}-a b^3-d e^3)^2\\
 +4 (u_{103}-a c^3-d f^3)^2+4 (u_{031}-a b^3 c-d e^3 f)^2+4 (u_{013}-a b c^3-d e f^3)^2 \\
 + 12 (u_{112}-a b c^2-d e f^2)^2+12 (u_{211}-a b c-d e f)^2+12 (u_{121}-a b^2 c-d e^2 f)^2.
 \end{array}
$$
The set of complex critical points is the zero locus of the ideal
$$I \quad = \quad \left\langle \frac{\partial G}{\partial a},\frac{\partial G}{\partial b},\frac{\partial G}{\partial c},\frac{\partial G}{\partial d},\frac{\partial G}{\partial e},\frac{\partial G}{\partial f}\right\rangle : 
\bigl(\,\langle a d \rangle\cap \langle b-e, c-f\rangle \, \bigr)^\infty.$$
For applications, we are interested in the real points in this variety.

{\em Computational proof of Proposition \ref{prop:195}.}
As argued in \cite[\S 2]{DHOST},
the ideal $I$ is radical and zero-dimensional
when the $u_{ijk}$ are generic rational numbers.
The number of solutions is the degree of $I$,
and we found this to be $ 370 = 2 \cdot 195$.
This is twice the ED degree of $\mathcal{L}_{\leq 2}$
with respect to $\Lambda = \Theta$. 
For this computation we used the {\tt FGb}
library in {\tt maple}. We used Gr\"obner bases over the finite field ${\sf
  GF}(65521)$ to avoid the swelling of rational coefficients, the data
$u_{ijk}$ are chosen uniformly at random in this field, and we
saturate only by $\langle a d (b-e)\rangle$.
The computation took 90 seconds and returned 390 critical points of $G$. 
Performing the same computation with the coefficients $1,6,4,12$ in $G(a,b,c,d,e,f)$
replaced with random field elements, we find  $3626 = 2 \cdot 1813 $ critical points, and hence
${\rm EDdegree}_{\rm gen}(\mathcal{L}_{\leq 2}) = 1813$.
\endproof

\begin{example} \label{ex:schultz2} \rm
We return to the particular data set in
Example \ref{ex:schultz1}.  Using the above parametrization, the best
rank 2 approximation can be obtained by solving a polynomial system.
This can be achieved by using symbolic or numerical
methods.

A numerical computation conducted by Jose Rodriguez with the software
{\tt Bertini} indicates that, for Thomas Schultz' data, precisely $9$
of the $195$ critical points are real. These correspond to $2$ local
minima and $7$ saddle points of the Euclidean distance
function. The precomputation with generic data took 2  hours on
  40 {\tt AMD Opteron 6276/2.3Ghz} cores. Then the computation
  with the numerical data in Example \ref{ex:schultz1} was achieved in 1 minute.
 
 These results were also computed by symbolic methods: a Gr\"obner basis
 computation conducted by Jean-Charles Faug\`ere and Mohab Safey El
 Din with the software {\tt FGb} returned an algebraic parametrization of the $195$ complex critical points
 by the roots of a univariate polynomial of degree $195$. 
 This  polynomial has $9$ real roots. Two of them correspond to the two
 local minima. The average size of the integer coefficients of this univariate
 polynomial is $11000$ digits.
  For this computation, the above formulation as an unconstrained optimization problem was used. 
 It took 11 minutes on a  2.6GHz {\tt IntelCore i7}.
  In general, for symbolic methods,
   unconstrained formulations seem to be better
     than the general implicit formulation in Proposition~\ref{prop:generalformulation}.
     See the comparisons of timings in Table~\ref{table:EDsecdet}.
     However, most instances of (\ref{eq:frobnorm2})
     do not admit an unconstrained formulation,
     because $\mathcal{L}_{\leq r}$ is usually not unirational. \hfill $\diamondsuit $\end{example}

Our last topic in this section is the study of Sylvester matrices.
We consider two arbitrary polynomials $F$ and $G$ in one variable $t$.
Suppose their degrees are $m$ and $n$ with $m \leq n$, so
$$ F(t) \,=\,  \sum_{i=0}^m a_i  t^i \quad \hbox{and} \quad
G(t)  \,=\, \sum_{j=0}^n b_j  t^j . $$
Fix $k $ with $ 1\leq k \leq m$.
The $k$-th Sylvester matrix of the pair $(F,G)$ equals
$$
{\rm Syl}_k(F,G) \quad = \quad
\begin{bmatrix}
a_0&  0    & \cdots &  0 &   b_0&  0    & \cdots &  0 \\
\vdots &   a_0 & \ddots&  \vdots &  \vdots &   b_0 & \ddots&  \vdots \\
a_m &   \vdots & \ddots & 0 &       b_n &   \vdots & \ddots & 0 \\
0 &          a_m  &   \vdots & a_0 &     0 &          b_n  &   \vdots & b_0 \\
\vdots &  \vdots   & \ddots & \vdots &  \vdots &  \vdots   & \ddots & \vdots \\
0 &          0  &   \cdots & a_m &       0 &          0  &   \cdots & b_n 
\end{bmatrix}
$$
This matrix has  $n+k$ rows and
$n-m+2k$ columns, so it is square for $k=m$,
and it has more rows than columns for $k < m$.
The maximal minors have size $n-m+2k$,
and they all vanish when
${\rm Syl}_k(F,G)$ has a non-zero
vector in its kernel. Such a vector
corresponds to a polynomial of degree $m-k+1$ that is a common factor
of $F$ and $G$.

The {\em approximate gcd problem} in computer algebra \cite{KL, KYZ}
aims to approximate a given pair $(F,G)$
by a nearby pair $(F^*,G^*)$ whose
Sylvester matrix ${\rm Syl}_k(F^*,G^*)$
has linearly dependent columns.
Writing $\mathcal{L}$ for the subspace of 
Sylvester matrices,  this is precisely our
ED problem for  $ \mathcal{L}_{\leq n-m+2k-1}$.
The following theorem furnishes a formula
for ${\rm EDdegree}_{\rm gen}(\mathcal{L}_{\leq n-m+2k-1})$.

\begin{theorem} \label{thm:FG}
For the variety of pairs $(F,G)$ of univariate polynomials
of degrees $(m,n)$ with a common factor of degree $m{-}k{+}1$,
the generic ED degree equals that of the
Segre variety of $(m{-}k{+}2) \times (n{-}m{+}2k)$-matrices of rank $1$.
It is given by  setting $s = 0$ in (\ref{eq:EDdegdual}). Using the {\tt Macaulay2} function {\tt ED} in
Example \ref{ex:1350}, we can write this ED degree~as
$$\,{\rm EDdegree}_{\rm gen}(\mathcal{L}_{\leq n-m+2k-1}) \,\,= \,\,
{\tt ED(m-k+2,n-m+2*k,1,0)} . $$
\end{theorem}

\begin{proof}
A natural desingularization is given by multiplying with the
desired common factor:
\begin{equation}
\label{eq:segremap}
\begin{matrix}
 \PP^{m-k+1} \times \PP^{n-m+2k-1} & \to &  \mathcal{L}_{\leq n-m+2k-1}, \\
\left[\,A(t)\,,\,\left(B(t),C(t)\right)\,\right] & \mapsto & \left[\,A(t)B(t),A(t)C(t) \,\right].
\end{matrix}
\end{equation}
Here $A(t), B(t),C(t)$  are polynomials of degrees
$m-k+1, k-1,n-m+k-1$ respectively.
The map (\ref{eq:segremap})
lifts to a  linear projection map
from the Segre embedding of $ \PP^{m-k+1} \times \PP^{n-m+2k-1}$.
Work of Piene \cite[\S 4]{Pie} implies that
the degrees of polar loci  can be computed
on that Segre variety. 
The ED degree is a sum of degrees of these,
by Corollary \ref{thm:sEDpolar}. 
The result follows.
\end{proof}

For $m = k$, when the Sylvester matrix is square,
Theorem \ref{thm:FG} refers to
$2 \times (n{+}m)$-matrices of rank $1$.
Similarly to \cite[Example 5.12]{DHOST},
their ED degree is $4(m{+}n)-2$.

\begin{corollary}
The generic ED degree of the
Sylvester determinant ${\rm Syl}_m $ equals $4(m+n)-2$.
\end{corollary}

We consider three
natural choices of weight matrices for the
low-rank approximation of  Sylvester matrices.
As before in Table \ref{table:EDHankel}, we write
 $\Omega_{m,n}$
 for the weight matrix
that represents the Euclidean distance on $\R^{m+n+2}$: it is the matrix which has the same pattern as ${\rm Syl}_k$ with $a_i$ and $b_j$ replaced respectively by $1/(n-m+k)$ and $1/k$. 
We also write $\Theta_{m,n}$ for the weight matrix of
the rotation invariant quadratic form: $a_i$ is replaced by $1/((n-m+k)\binom{m}{i})$ and $b_j$ is replaced by $1/(k\binom{n}{j})$. In Table \ref{table:EDSylvester} we present the
ED degrees  for these choices of weights.
The left table shows 
the generic behavior predicted by Theorem \ref{thm:FG}.
At present, we do not
know a general formula for the entries of the two tables on the right side, but
we are hopeful that an approach like \eqref{eq:3ED}
will lead to such formulas.
Along the rightmost margins,
where the matrix ${\rm Syl}_m $ is square,
the formula seems to be
$\,{\rm EDdegree}_{\Theta}( \mathcal{L}_{\leq n+k-1})\,= \, 2n$.

\begin{table}[h]
\centering
\begin{tabular}{@{}|@{\,}c@{\,\,}|@{\,}c@{\,}|@{\,}c@{\,}|@{\,}c@{\,}|@{\,}c@{\,}|@{}}
\hline
\multicolumn{5}{|c|}{
$\Lambda$ is generic
}\\
\hline
$(m,n) \backslash k \! $ &$1$&$2$&$3$&$4$\\
\hline\hline
$\!\! (2,2)$& 10 & 14 & & \\
$\!\! (2,3)$& 39 & 18 & & \\
$\!\! (2,4)$& 83 & 22 & & \\
$\!\! (2,5)$& 143 & 26  & &  \\
$\!\! (3,3)$& 14 & 83 & 22 & \\
$\!\! (3,4)$& 83 & 143 & 26  &\\
$\!\! (3,5)$& 284 & 219 & 30 &\\
$\!\! (4,4)$& 18 & 284  & 219   & 30  \\
$\!\! (4,5)$& 143 & 676  & 311 & 34 \\
\hline
\end{tabular} 
\begin{tabular}{@{}|@{\,}c@{\,\,}|@{\,}c@{\,}|@{\,}c@{\,}|@{\,}c@{\,}|@{\,}c@{\,}|@{}}
\hline
\multicolumn{5}{|c|}{$\Lambda = \Omega_{m,n}$
}\\
\hline
$(m,n) \backslash k \! $ &$1$&$2$&$3$&$4$\\
\hline\hline
$\!\! (2,2)$& 2 & 6 & & \\
$\!\! (2,3)$& 23 & 18 & & \\
$\!\! (2,4)$& 75 & 22 & & \\
$\!\! (2,5)$& 119 & 18  & &  \\
$\!\! (3,3)$& 2 & 19 & 10 & \\
$\!\! (3,4)$& 35 & 95 & 26  &\\
$\!\! (3,5)$& 188 & 203 & 26 &\\
$\!\! (4,4)$& 2 & 36 & 59 & 14  \\
$\!\! (4,5)$& 47 & 276 & 215 & 34 \\
\hline
\end{tabular}
\begin{tabular}{@{}|@{\,}c@{\,\,}|@{\,}c@{\,}|@{\,}c@{\,}|@{\,}c@{\,}|@{\,}c@{\,}|@{}}
\hline
\multicolumn{5}{|c|}{$\Lambda = \Theta_{m,n}$}\\
\hline
$(m,n) \backslash k \! $ &$1$&$2$&$3$&$4$\\
\hline\hline
$\!\! (2,2)$&2 & 4 & & \\
$\!\! (2,3)$& 19 & 6 & & \\
$\!\! (2,4)$& 29 & 8 & & \\
$\!\! (2,5)$& 61 &  10 & &  \\
$\!\! (3,3)$& 2 & 19 &  6 & \\
$\!\! (3,4)$& 41 & 53 &   8 &\\
$\!\! (3,5)$& 106 & 81 &  10 &\\
$\!\! (4,4)$& 2 & 50  & 45   & 8   \\
$\!\! (4,5)$& 71 & 256  & 101 &  10 \\
\hline
\end{tabular}
\caption{\label{table:EDSylvester}
Weighted ED degrees for Sylvester matrices ${\rm Syl}_k(F,G)$}
\end{table}

\bigskip
\medskip

\noindent
{\bf Acknowledgements.}\\
We thank the following colleagues for their help with this project:
Jean-Charles Faug\`ere, William Rey,
Ragni Piene,  Jose Rodriguez,
 Mohab Safey El Din, \'Eric
Schost,  and Thomas Schultz.
Giorgio Ottaviani is a member of GNSAGA-INDAM. 
Pierre-Jean Spaenlehauer and Bernd Sturmfels were hosted by
the Max-Planck Institute f\"ur Mathematik in Bonn, Germany.
Bernd Sturmfels was also supported
by the NSF (DMS-0968882).


\begin{thebibliography}{10}


\bibitem{BHSW} {\sc D.~Bates, J.~Hauenstein, A.~Sommese and
C.~Wampler}, {\sl Numerically Solving Polynomial Systems                    
with Bertini}, SIAM, 2013.

\bibitem{BPT} {\sc G.~Blekherman, P.~Parrilo and R.~Thomas},
{\sl Semidefinite Optimization and Convex Algebraic Geometry},
MOS-SIAM Series on Optimization 13, SIAM, Philadelphia, 2013.

\bibitem{BR} {\sc H.-C.G.~von Bothmer and K.~ Ranestad}, {\em A general formula for the algebraic degree in semidefinite programming},
Bull. London Math. Soc. {\bf 41} (2009) 193--197.

\bibitem{CHKS}
{\sc F.~Catanese, S.~Ho\c sten, A.~Khetan and B.~Sturmfels},
 {\em The maximum likelihood degree}, American J.~Math.
 {\bf 128} (2006) 671-697. 
 
\bibitem{CFP}
{\sc M.~Chu, R.~Funderlic, and R.~Plemmons},
{\em  Structured low rank approximation},
  Linear Algebra Appl. {\bf 366} (2003) 157--172.

\bibitem{DHOST}
{\sc J.~Draisma, E.~Horobe\c{t}, G.~Ottaviani,  B.~Sturmfels, and R.~Thomas},
{\em The Euclidean distance degree of an algebraic variety},   {\tt arXiv:1309.0049}.

\bibitem{DK} {\sc I.~Dolgachev and M.~Kapranov}, {\em Arrangement of hyperplanes and vector bundles
on ${\mathbb P}^n$}, Duke Math.~J.~{\bf 71} (1993) 633--664. 

\bibitem{Fau02}
{\sc J.-C.~Faug\`ere},
{\em A new efficient algorithm for computing Gr\"obner bases without reduction to zero (F5)}, in Proceedings of ISSAC 2002, 75--83.
\\ FGb library available at {\tt http://www-polsys.lip6.fr/\textasciitilde jcf/Software/FGb/}.

\bibitem{FS}
{\sc S.~Friedland and M.~Stawiska},
{\em Best approximation on semi-algebraic sets and k-border rank approximation of symmetric tensors},
{\tt arXiv:1311.1561}.
	
\bibitem{Fulton} {\sc W.~Fulton}, {\em Intersection Theory}, Springer,
Berlin, 1998.

\bibitem{GG} {\sc M.~Golubitsky and V.~Guillemin},
{\em Stable Mappings and their Singularities}, Springer-Verlag, New York, 1974.

\bibitem{M2}
{\sc D.R.~Grayson and M.E.~Stillman},
{\em Macaulay2, a Software System for Research in Algebraic Geometry}. Available at
{\tt http://www.math.uiuc.edu/Macaulay2/}.


\bibitem{Holme} {\sc A.~Holme}, {\em The geometric and numerical properties of duality in projective algebraic geometry}, Manuscripta Math. {\bf 61}  (1988) 145--162.

\bibitem{KL} {\sc N.K.~Karmarkar and Y.N.~Lakshman},
{\em On approximate GCDs of univariate polynomials},
J.~Symbolic Comput.~{\bf 26} (1998) 653--666.

\bibitem{KYZ} {\sc E.~Kaltofen, Z.~Yang and L.~Zhi},
{\em Structured low rank approximation of a Sylvester matrix},
in: D.~Wang, L.~Zhi (Eds.): Symbolic-Numeric Computation,
Trends in Mathematics, Birkh\"auser, 2007, pp.~69--83.

\bibitem{Lan} {\sc J.M.~Landsberg},
{\em Tensors: Geometry and Applications},
Graduate Studies in Mathematics, {\bf 128},
American Math.~Society, Providence, 2012.

\bibitem{MMH} {\sc J.H.~Manton, R.~Mahony and Y.~Hua},
{\em The geometry of weighted low-rank approximation},
IEEE Transactions on Signal Processing {\bf 51} (2003) 500--514.

\bibitem{Mar} {\sc I.~Markovski}, {\em Structured low-rank approximation and its applications},
Automatica {\bf 44} (2008), no. 4, 891--909.

\bibitem{Pie} {\sc R.~Piene}, {\em Polar classes of singular varieties},
Ann. Sci. \'Ecole Norm. Sup. (4) {\bf 11} (1978) 247--276.

\bibitem{Rai} {\sc C.~Raicu},
{\em Secant varieties of Segre-Veronese varieties},
Algebra and Number Theory {\bf 6} (2012) 1817--1868.

\bibitem{Ran} {\sc K.~Ranestad},  Algebraic degree in semidefinite and polynomial optimization,  in J.-B.~Lasserre and M.~Anjos (eds.):
{\em  Handbook on Semidefinite, Conic and Polynomial Optimization},
  Springer, 2012, pp.~61-75.
  
\bibitem{Rey} {\sc W.~Rey}, {\em On weighted low-rank approximation}, {\tt arXiv:1302.0360}.

\bibitem{Sch}
{\sc T.~Schultz, A.~Fuster, A.~Ghosh, R.~Deriche, L.~Florack, and L.-H.~Lim},
{\em Higher-order tensors in diffusion imaging},
In: Visualization and Processing of Tensors and Higher Order Descriptors for Multi-Valued Data,
Springer, 2013.

\bibitem{SJ} {\sc N.~Srebro and T.~Jaakkola}, {\em Weighted low-rank approximations},
International Conference on Machine Learning (2003) 720--727.

\end{thebibliography}
\end{document}